\documentclass[10pt, reqno]{amsart}
\usepackage{amssymb,latexsym,amsmath,amsfonts, amsthm}
\usepackage{latexsym}
\usepackage[numbers,sort&compress]{natbib}
\usepackage{color}
\usepackage[mathscr]{eucal}
\usepackage{comment}
\usepackage{mathtools}
\usepackage{graphicx}
\usepackage{subcaption}
\usepackage{float}
\usepackage{enumerate}

\usepackage[linkcolor=blue, citecolor=red]{hyperref}
\hypersetup{colorlinks=true, urlcolor=blue}
\usepackage[nameinlink]{cleveref}

\numberwithin{equation}{section}

\theoremstyle{definition}
\newtheorem{definition}{Definition}[section]

\theoremstyle{remark}
\newtheorem{remark}[definition]{Remark}
\theoremstyle{plain}
\newtheorem{theorem}[definition]{Theorem}
\newtheorem{lemma}[definition]{Lemma}
\newtheorem{proposition}[definition]{Proposition}

\newcommand{\tl}{\tilde}
\newcommand{\lt}{L^2(0, 2\pi)}
\newcommand{\tor}{(0, 2\pi)}
\newcommand{\intg}{\int_{0}^{2\pi}}

\newcommand{\z}{\mathbb{Z}}

\newcommand{\zahl}{\mathbb{Z}}
\newcommand{\nat}{\mathbb{N}}

\newcommand{\ip}[2]{\left<{#1},{#2}\right>}
\newcommand{\norm}[1]{\left\|#1\right\|}
\newcommand{\abs}[1]{\left|#1\right|}

\date{\today}

\newcommand{\dist}{{\rm dist}}

\newcommand{\cplx}{\mathbb{C}}

\newcommand{\rea}{\mathbb{R}}

\begin{document}
	\title[Control of the stabilized Kuramoto-Sivashinsky]{Null controllability of the linear Stabilized Kuramoto-Sivashinsky System using moment method}
	\author[Manish Kumar and Subrata Majumdar]{Manish Kumar$^{\dagger,1}$ and Subrata Majumdar$^{\dagger,2}$$^*$ }
	\thanks{$^{\dagger}$Indian Institute of Science Education and Research Kolkata, Campus road, Mohanpur, West Bengal 741246, India;
		$^{1}$\url{mk19ip001@iiserkol.ac.in}, $^{2}$\url{sm18rs016@iiserkol.ac.in}}
	\address{Department of Mathematics and Statistics, Indian Institute of Science Education and Research Kolkata,
		Mohanpur -- 741 246}
	\email{mk19ip001@iiserkol.ac.in, sm18rs016@iiserkol.ac.in}
	\thanks{$^*$Corresponding author}
	\thanks{Manish Kumar is supported by Prime Minister Research Fellowship, Govt. of India and Subrata Majumdar is supported by Department of Atomic Energy  and National Board for Higher Mathematics fellowship.}
	\begin{abstract}
		This paper deals with the null controllability of a coupled parabolic system, which is the Kuramoto-Sivashinsky-Korteweg-de Vries equation coupled with heat equation through first order derivatives. More precisely, we prove the null controllability of the system with a single localized bilinear interior control acting on either of the components of the coupled system, and with a single periodic boundary control acting through zeroth order derivatives of either of the components. We employ the well-known moment method to study the controllability of the concerned system.
	\end{abstract}
	\maketitle
	\tableofcontents
	\section{Introduction and main result} 
	\subsection{Setting of the problem}
	One dimensional Kuramoto-Sivanshinsky (KS in short) equation, which reads as
	\begin{equation}\label{KS}
		u_t+\gamma u_{xxxx}+au_{xx}+uu_x=0,
	\end{equation}
	appears in many physical phenomena like phase turbulence and wave propagation in reaction-diffusion systems (see \cite{Kur78}, \cite{Kur75}, \cite{Kur76}),  flame front propagation (see \cite{Siv77}), where $\gamma>0, a>0$ are coefficients accounting for the long-wave instabilities and the short wave dissipation, respectively. Benney in \cite{ben66} added the KdV term $u_{xxx}$ in KS equation \eqref{KS} to include dispersive effects in the system. The resultant equation
	\begin{equation}\label{KS-KdV}
		u_t+\gamma u_{xxxx}+u_{xxx}+au_{xx}+uu_x=0,
	\end{equation} is called Kuramoto-Sivashinsky-Korteweg-de Vries (KS-KdV) equation, which is used to study a wide range of nonlinear dissipative waves. The solitary-pulse solutions of the KS-KdV equation \eqref{KS-KdV} are found to be unstable and so in \cite{MFK01}, the authors coupled the KS-KdV equation with an extra-dissipative equation (heat equation) in one dimensional setup, which resulted in the existence of stable solitary-pulses due to the combination of dissipative and dispersive features. For this reason, the system is also known as stabilized Kuramoto-Sivashinsky equation, and is given by
	\begin{equation}\label{KS-KdV_heat}
		\begin{cases}
			u_t+\gamma u_{xxxx}+u_{xxx}+au_{xx}+uu_x=v_x,\\
			v_t-\Upsilon  v_{xx}+\nu v_x=u_x,
		\end{cases}
	\end{equation}
	where, $\Upsilon>0$ is the dissipative parameter (effective diffusion coefficient), and $\nu\in \rea\setminus\{0\}$ is the group velocity mismatch between the wave modes.
	
	Let $T>0$ be any real number. This paper studies the null controllability of the linearized version of the aforementioned coupled system \eqref{KS-KdV_heat}, posed in $(0,T)\times (0,2\pi)$ with periodic boundary conditions. For our study, we set all the system parameters equal to $1$. More precisely, we consider the following system:
	\begin{equation}\label{eq:main}
		\begin{cases}
			u_t+ u_{xxxx}+u_{xxx}+u_{xx}=v_x,& (t,x)\in(0,T)\times(0,2\pi),\\
			v_t- v_{xx}+ v_x=u_x,& (t,x)\in(0,T)\times(0,2\pi),\\
			u(t,0)=u(t,2\pi),u_x(t,0)=u_x(t, 2\pi), & t\in (0,T),\\
			u_{xx}(t,0)=u_{xx}(t, 2\pi), u_{xxx}(t,0)=u_{xxx}(t, 2\pi), & t\in (0,T),\\
			v(t,0)=v(t,2\pi), v_x(t,0)=v_x(t, 2\pi), & t\in (0,T),\\
			u(0,x)=u_0(x), v(0,x)=v_0(x), & x \in (0, 2\pi).
		\end{cases}
	\end{equation}
	
	In order to define the notion of null controllability, let us consider the following general control system corresponding to the above system:
	\begin{equation}\label{eq:cntrl}
		\begin{cases}
			u_t+ u_{xxxx}+u_{xxx}+u_{xx}=v_x+h_1,& (t,x)\in(0,T)\times(0,2\pi),\\
			v_t- v_{xx}+ v_x=u_x+h_2,& (t,x)\in(0,T)\times(0,2\pi),\\
			u(t,0)=u(t,2\pi)+q_1(t),u_x(t,0)=u_x(t, 2\pi), & t\in (0,T),\\
			u_{xx}(t,0)=u_{xx}(t, 2\pi), u_{xxx}(t,0)=u_{xxx}(t, 2\pi), & t\in (0,T),\\
			v(t,0)=v(t,2\pi)+q_2(t), v_x(t,0)=v_x(t, 2\pi), & t\in (0,T),\\
			u(0,x)=u_0(x), v(0,x)=v_0(x), & x \in (0, 2\pi),
		\end{cases}
	\end{equation}
	where the functions $h_1(x,t),h_2(x,t)$ and $q_1(t),q_2(t)$ are interior and boundary controls, respectively. 
	\begin{definition}[\textbf{Interior null controllability}]
		Let $T>0$ and let $X$ be a Hilbert space. Then system \eqref{eq:main} is said to be interior null controllable in $X$ at time $T$, if for any given initial data $(u_0,v_0)\in X$ there exists $h_1$ $(\text{or }h_2)\in L^2(0,T;L^2(0,2\pi))$ such that the solution of \eqref{eq:cntrl} with $h_2$ $(\text{or }h_1),q_1, q_2$ identically equal to $0$ satisfies $u(T,\cdot)=0$ and $v(T,\cdot)=0$.
	\end{definition}
	\begin{definition}[\textbf{Boundary null controllability}]
		Let $T>0$ and let $X$ be a Hilbert space. Then system \eqref{eq:main} is said to be boundary null controllable in $X$ at time $T$, if for any given initial data $(u_0,v_0)\in X$ there exists $q_1$ $(\text{or }q_2)\in L^2(0,T)$ such that the solution of \eqref{eq:cntrl} with $h_1,h_2,q_2$ $(\text{or }q_1)$ identically equal to $0$ satisfies $u(T,\cdot)=0$ and $v(T,\cdot)=0$.
	\end{definition}
	
	In this paper, we establish null controllability of the system \eqref{eq:main} with exactly one control, i.e., only one of the control functions will be present in the system \eqref{eq:cntrl} and the rest will be assumed to be identically $0$.

	\subsection{Related literature and Motivation for the problem}
	Controllability of parabolic partial differential equations (PDEs) has gained a lot of interest among researchers throughout the years. Several methods have been introduced so far to study such problems. 
	Let us mention some pioneer works among them. Null controllability of one-dimensional heat equation was first investigated by Fattorini and Russell in \cite{FR71}, where the method of moments was introduced to study such problem. Later on, Lebeau and Robbiano demonstrated local Carleman estimate for the heat equation in higher dimensional setting in \cite{LR95}. Applying global Carleman estimate, Fursikov and Immanuvilov  studied the null controllability of heat equation in one dimension in \cite{F96}. Transmutation method was given by Miller in \cite{LM06} to deal with the same problem of heat equation in one dimension. Also, a constructive method named as Flatness method was proposed by Martin, Rosier, and Rouchon in \cite{LR16}. Recently, Coron and Nguyen developed the backstepping method in \cite{Co17arma} to conclude the null controllability of one-dimensional heat equation. For more results regarding null controllability of parabolic equation, see \cite{AG08} and references therein.
	
	Now we explore the existing literature for the KS equation. Because of the physical importance of stability, it was the first thing to be studied for KS equation in the field of control theory, for example, see \cite{AC00}, \cite{AC000}, \cite{John96}. Let us also mention the recent works  \cite{Kirs18}, \cite{Cer19}, concerning the stability of KS. The authors in \cite{Kirs18} gave general results, explaining when one can conclude about the stability of a nonlinear system from the stability of the corresponding linearized system, and further illustrated in the case of KS equation. 
	The first work regarding controllability of KS equation was done in \cite{EC10}, where the author 
	studied the boundary null controllability of linear KS equation with a single control force acting on first order derivative at left end point, using moment method. 
	Stability of the system was also studied in this paper. The authors in \cite{EC11} considered  nonlinear KS equation and proved its controllability to trajectory with Dirichlet boundary control acting at left endpoint of zeroth and first order derivative. They first used Carleman estimates to study the null controllability of linearized KS equation and then used local inversion theorem to get the desired result for nonlinear KS equation. The study of  null controllability of linear KS equation was further developed in \cite{EC17}, wherein the authors 
	studied the boundary null controllability of linear KS equation, but now with the control acting only on zeroth order derivative at left end point, using moment method. 
	Next, they considered the Neumann boundary case and proved that the linear KS system is not null controllable with a single control acting on either of the second or third order derivatives but is so with both the boundary controls acting simultaneously on the system. Further, they also proved the null controllability of linear KS system with localized interior control. 
	In \cite{Gao18}, \cite{GCL15}, the Carleman estimates  for linear stochastic KS equation has been explored. The authors in \cite{Tak17} studied the local boundary null controllability of KS equation in one and two dimensional setup. The controllability of state-constrained linear stochastic KS system has been studied in the works \cite{Gao15}, and \cite{Gao20}. Results regarding the global exact controllability to the trajectory of KS has been given in \cite{Gao20(1)}. 
	
	After single parabolic equation, the study of controllability of coupled parabolic system has fascinated a lot of control theorist in the last two decades, the  main reason being its wide range of appearance in different practical situations, such as in the study of physical phenomena, chemical reactions and  in a wide variety of mathematical biology. The coupled nonlinear system \eqref{KS-KdV_heat}, whose linearized version is being considered in  our study, describes the surface waves on multi-layered liquid films. For more details on the controllability of linear coupled parabolic systems, one can have a look into a very nice survey report \cite{Ter11} and references therein. For the existing controllability results of the considered system \eqref{KS-KdV_heat}, let us mention \cite{EC16}, \cite{EC12} and \cite{EC15}. In \cite{EC12}, the authors proved local null controllability of \eqref{KS-KdV_heat} with an extra $v^2$ term in the second equation with Dirichlet boundary conditions using three boundary controls whereas, in \cite{EC15} and \cite{EC16} the authors proved local null controllability and local controllability to the trajectories of system \eqref{KS-KdV_heat} with localized interior control acting on KS-KdV equation and heat equation, respectively. In all of these three works, different Carleman inequalities have been proved, depending on the system, to get the controllability result for corresponding linearized systems, and then the desired result for the nonlinear systems have been proved using local inversion theorems. 
	The authors in \cite{EC19} considered a simplified version of linear stabilized KS system with coupling through zeroth order derivative in KS equation, wherein they gave positive as well as negative results of null controllability depending on the coefficient $\nu$ using method of moments. Using the source term method, this has been extended to the corresponding nonlinear system (i.e., with $uu_x$ term) in \cite{hernandezsantamaria:hal-03090716}, which gives null controllability of the system under the condition of $\sqrt{\nu}$ being quadratic irrational. Let us also mention the most recent works \cite{BV22} and \cite{Vic22} which deals with an insensitizing control problem and controllability of stochastic stabilized KS system, respectively. Lastly, we mention the work \cite{Vic21} which deals with the controllability issues of stabilized KS system in numerical setup by discretizing the time variable.
	
	To the best of the authors' knowledge, the present work is the first one to prove null controllability of the linear stabilized KS system \eqref{eq:main}  using the moment method. This was possible here because of the periodic boundary conditions considered in the system, which gave nice behavior of the eigen-elements.

	\subsection{Notations and functional setting}\label{sect1.4}
	Let us denote the Lebesgue space of  complex valued square integrable periodic functions over $[0,2\pi]$ by $\lt$. Further, we denote the space $L^2(0,2\pi)\times L^2(0,2\pi)$ by $\mathbf{Z}$ with the usual Hermitian product
	\begin{equation*}
		\ip{\begin{pmatrix}
				u_1\\ v_1
		\end{pmatrix}}{\begin{pmatrix}
				u_2\\ v_2
		\end{pmatrix}}_{\mathbf{Z}}=\intg u_1(x) \overline{u_2(x)} dx+ \intg v_1(x) \overline{v_2(x)} dx.
	\end{equation*}
	For $s\geq0$, we define the classical periodic Sobolev space $H^s_{per}(0,2\pi)$ as the subspace of $\lt$, consisting functions of the form 
	\begin{equation*}
		\phi=\sum_{k\in \z}\phi_k e^{-ikx},
	\end{equation*} such that \begin{equation*}
		\sum_{k\in \z}\left(1+|k|^{2}\right)^s |\phi_k|^2 < \infty.
	\end{equation*}
	Note that $\phi_k=\intg \phi(x) e^{ikx} dx, k \in \mathbb{Z}.$
	Thus, $H^s_{per}(0,2\pi)$ is a Hilbert space endowed with the following norm
	\begin{equation*}
		\norm{\phi}_{H^s_{per}\tor}=\left(\sum_{k\in \z} \left(1+|k|^{2}\right)^s |\phi_k|^2\right)^{\frac{1}{2}}.
	\end{equation*}
	Further, for $s\in \nat$, we also write 
	\begin{align*}
		H^s_{per}(0,2\pi)=\left\{f\in H^s(0,2\pi):f^{i}(0)=f^{i}(2\pi), i \in \{0,1,\ldots,s-1\}\right\}.
	\end{align*} 
	
	Let $X$ be any Lebesgue space defined on $(0,2\pi)$. We will often use the shorthand notation $L^2(X)$ and $H_{per}^s$  to denote the space $L^2(0,T;X)$ and $H_{per}^s(0,2\pi)$, respectively. Also, we will use the alphabets $c,C$  to denote positive generic constants.

	Let $U=\begin{pmatrix}
	u \\ v
	\end{pmatrix}$, then the above system can be written in infinite dimensional ODE setup as \begin{align}
		U'(t)=AU(t),
	\end{align} where the operator $A$ can be written in the following form:
	\begin{equation}\label{A}
		A= \begin{pmatrix}
			- \dfrac{d^4}{dx^4}-\dfrac{d^3}{dx^3}- \dfrac{d^2}{dx^2} &  &\dfrac{d}{dx}\\ \\
			\dfrac{d}{dx} & & \dfrac{d^2}{dx^2}-\dfrac{d}{dx} 
		\end{pmatrix} ,
	\end{equation}
	with its domain 
	\begin{align*}
		D(A)=H^4_{{per}}\tor\times {H}^2_{{per}}\tor.
	\end{align*}

	Let $\mathbf{H}=H^2_{per}(0,2\pi)\times H^1_{per}(0,2\pi)$.
	We denote the dual space of $\mathbf{H}$ with respect to the pivot space $\mathbf{Z}$ by $\mathbf{H}^*=(H^2_{per}(0,2\pi))^*\times (H^1_{per}(0,2\pi))^*$, where $(H^k_{per}(0,2\pi))^*$ is the dual space of $H^k_{per}(0,2\pi)$ with respect to the pivot space $L^2(0,2\pi)$, for $k=1,2$.	This $\mathbf{H}^*$ is the state space in our study and we express the duality product between $\mathbf{{H}}^*$ and $\mathbf{H}$ by $\ip{\cdot}{\cdot}_{\mathbf{{H}}^*,\mathbf{H}}$.

	\subsection{Main results}
	
	We now state the main results of this paper, concerning the null controllability of the system \eqref{eq:main} using localized interior and boundary control acting at different positions. 
	
	\noindent
	We first consider the following control system:
	\begin{equation}\label{eq:cntrl1}
		\begin{cases}
			u_t+ u_{xxxx}+u_{xxx}+u_{xx}-v_x= h(x,t) & (t,x)\in (0,T)\times(0, 2\pi),\\
			v_t- v_{xx}+ v_x-u_x=0 & (t,x)\in (0,T)\times(0, 2\pi),\\
			u(0,x)=u_0(x), v(0,x)=v_0(x) & x \in (0, 2\pi).
		\end{cases}
	\end{equation}
	with the same boundary condition as in \eqref{eq:main}, where $h$ denotes the control function. 
	\begin{theorem}\label{thm1}
		For any time $T>0$ and any $
		(u_0, v_0)
		\in (H_{per}^2)^*\times (H_{per}^1)^*$ with $\ip{v_0}{1}_{{(H^1_{per})}^*,H^1_{per}}=0$, there exists a localized (bilinear) interior control $h\in L^2((0,T)\times (0, 2\pi))$ of the form $h(x,t)=1_{\omega}f(x)g(t)$, where $\omega$  is any nonempty open subset of $(0,2\pi)$, such that the solution of \eqref{eq:cntrl1} satisfies $u(T,\cdot)=0$ and $v(T,\cdot)=0$. 
	\end{theorem}
	\noindent
	Next we consider the same system as in \eqref{eq:cntrl1} but now with the control $h$ acting on the heat equation. That is, we consider the system
	\begin{equation}\label{eq:cntrl1a}
		\begin{cases}
			u_t+ u_{xxxx}+u_{xxx}+u_{xx}-v_x= 0 & (t,x)\in (0,T)\times(0, 2\pi),\\
			v_t- v_{xx}+ v_x-u_x=h(x,t) & (t,x)\in (0,T)\times(0, 2\pi),\\
			u(0,x)=u_0(x), v(0,x)=v_0(x) & x \in (0, 2\pi).
		\end{cases}
	\end{equation}
	with the same boundary condition as in \eqref{eq:main}.
	\begin{theorem}\label{thm1a}
		For any time $T>0$ and any  $(u_0, v_0) \in (H_{per}^2)^*\times (H_{per}^1)^*$ with $\ip{u_0}{1}_{{(H^2_{per})}^*,H^2_{per}}=0$, there exists a localized (bilinear) interior control $h\in L^2((0,T)\times\tor)$ of the form $h(x,t)=1_{\omega}f(x)g(t)$, where $\omega$  is any nonempty open subset of $(0,2\pi)$, such that the solution of \eqref{eq:cntrl1} satisfies $u(T,\cdot)=0$ and $v(T,\cdot)=0$. 
	\end{theorem}
	\noindent
	Next we consider the following boundary control system: 
	\begin{equation}\label{eq:cntrl2}
		\begin{cases}
			u_t+ u_{xxxx}+u_{xxx}+u_{xx}-v_x=0, & (t,x)\in (0,T)\times(0, 2\pi),\\
			v_t- v_{xx}+ v_x-u_x=0,  & (t,x)\in (0,T)\times(0, 2\pi),\\
			u(t,0)=u(t,2\pi)+q(t), & t \in (0, T) \\
			u(0,x)=u_0(x), v(0,x)=v_0(x), & x \in (0, 2\pi).
		\end{cases}
	\end{equation}
	with rest of the boundary condition as in \eqref{eq:main}, where q is control.
	\begin{theorem}\label{thm2}
		Let $ (u_0, v_0)\in (H_{per}^2)^*\times (H_{per}^1)^*$ with $\ip{u_0}{1}_{{(H^2_{per})}^*,H^2_{per}}=0$. Then for any $T>0$, there exists $q\in L^2(0,T)$ such that the solution of \eqref{eq:cntrl2} satisfies $u(T,\cdot)=0$ and  $v(T,\cdot)=0$.
	\end{theorem}
	\noindent
	At last, we consider the boundary control system similar to \eqref{eq:cntrl2}, but now with the control acting at the periodic boundary of heat component, i.e.,
	\begin{equation}\label{eq:cntrl2a}
		\begin{cases}
			u_t+ u_{xxxx}+u_{xxx}+u_{xx}-v_x=0, & (t,x)\in (0,T)\times(0, 2\pi),\\
			v_t- v_{xx}+ v_x-u_x=0, & (t,x)\in (0,T)\times(0, 2\pi),\\
			v(t,0)=v(t,2\pi)+q(t), &t\in(0,T),\\
			u(0,x)=u_0(x), v(0,x)=v_0(x), & x \in (0, 2\pi),
		\end{cases}
	\end{equation}
	with rest of the boundary condition as in \eqref{eq:main}, where q is control.
	\begin{theorem}\label{thm2a}
		Let $(u_0, v_0) \in (H_{per}^2)^*\times (H_{per}^1)^*$ with $\ip{u_0}{1}_{{(H^2_{per})}^*,H^2_{per}}=0$ and $\ip{v_0}{1}_{{(H^1_{per})}^*,H^1_{per}}=0$. Then for any $T>0$, there exists $q\in L^2(0,T)$ such that the solution of \eqref{eq:cntrl2a} satisfies $u(T,\cdot)=0$ and  $v(T,\cdot)=0$.
	\end{theorem}
	\subsection{Short description of the method used}
	The proof of all the above-mentioned theorems are similar and is based on the celebrated method of moments.
	
	This method completely depends on the eigen-elements of the underlying spatial operator.
	Roughly speaking about this method, one has to first find an identity equivalent to the null controllability of the considered system using the main system and its adjoint system. Further, using the eigenvectors of the associated adjoint operator as the terminal data of the adjoint system, one can derive a system of identities, moment problem. Thus, the question of null controllability boils down to the existence of a solution to the derived moment problem. Now, solving the moment problem is immediate once the existence of a family biorthogonal to a certain exponential family, involving the eigenvalues of the operator, with proper $L^2$- estimate is known. Thus, the main difficulty in solving the moment problem and hence getting the null controllability result for a system is to get the existence of desired biorthogonal family.
	
	The moment method along with its advantage of giving the explicit form of control, instead of just giving the existence of control, comes up with a limitation as well. Since this method completely depends on the spectral properties of the corresponding operator, it cannot be used to study the controllability of systems in higher dimensions in general.
	
	As we mentioned before, this method of reducing the problem of null controllability to moment problem was initially done by H.O. Fattorini, and D.L. Russel for linear parabolic equation in their work \cite{FR71}. They further generalized the result regarding the existence of the biorthogonal family of real exponentials in \cite{FR74}. 
	The construction of biorthogonal family has been further investigated for complex exponentials as well, for e.g., see \cite{fernandez2010boundary}, \cite{boy14}, \cite{ter14}, \cite{gon21}. All of these results require some conditions to be satisfied by the power of exponential.
	
	In our case, we need to get a family biorthogonal to $\{e^{-\lambda_k^\pm t}\}$, where $\lambda_k^\pm$ denote two parabolic branches of eigenvalues of the underlying spatial operator. Although, one can get the existence of such a family from \cite{ter14}, but over here we have employed the method of \cite{LR14} to construct the biorthogonal family. 
	The construction of biorthogonal family given in \cite{LR14} was for similar type of exponential family but with parabolic and hyperbolic branches of eigenvalues. Such branching of eigenvalues led to a new approach for biorthogonal construction and because of such nature of the eigenvalues, there was also a time constraint for the existence of biorthogonal. In this paper, we successfully employ this method for parabolic-parabolic nature of eigenvalues, and  are also able to remove the time constraint for the existence of biorthogonal, as expected. It is worth to mention that, the method of \cite{LR14} has been utilized in the papers \cite{SC15}, \cite{Tao}, \cite{gao}, \cite{Gao1}.
	\subsection{Organization of the work}
	\Cref{sect2} deals with the well-posedness results of the studied control systems. In \Cref{sect3}, we mainly find the eigenvalues and corresponding eigenvectors of adjoint of the underlying spatial operator associated to system \eqref{eq:main}. This section also contains the statement of the theorem concerning the existence of biorthogonal family. \Cref{sect4} and \Cref{sect5} are devoted to the proof of the main results of this paper. We next give the construction of the desired biorthogonal family in \Cref{sect6}. At last, we give the proof of well posedness result in Appendix C (\Cref{sect7}).
	\section*{Acknowledgments}  
	We would like to thank Dr. Rajib Dutta for reading our article carefully and for his valuable comments. We also thank Golam Mostafa Mondal for pointing out that an identity given in the proof of  Lemma 18 of \cite{AM21} does not hold for $x=0$ and helping us to compute the integral calculation of Appendix B (\Cref{secinq}). Manish kumar acknowledges the Prime Minister Research Fellowship Scheme, Govt. of India for financial support, and Subrata Majumdar acknowledges the Department of Atomic Energy and National Board for Higher Mathematics fellowship, Grant No. 0203/16(21)/2018-R\&D-II/10708.
	\section{Well Posedness}\label{sect2}
	\noindent
	In this section, we give the well-posedness results for all of the above-mentioned control systems. We define the solution of the boundary control systems by means of transposition scheme and then give the corresponding well-posedness result. 
	
	\noindent
	We first consider the following general control system:
	\begin{equation}\label{eq:main1}
	\begin{cases}
	u_t+ u_{xxxx}+u_{xxx}+u_{xx}=v_x+g_1,& (t,x)\in (0,T)\times(0, 2\pi),\\
	v_t- v_{xx}+ v_x=u_x+g_2, & (t,x)\in (0,T)\times(0, 2\pi),\\
	u(t,0)=u(t,2\pi)+q_1(t),u_x(t,0)=u_x(t, 2\pi), & t\in (0,T),\\
	u_{xx}(t,0)=u_{xx}(t, 2\pi), u_{xxx}(t,0)=u_{xxx}(t, 2\pi), & t\in (0,T),\\
	v(t,0)=v(t,2\pi)+q_2(t), v_x(t,0)=v_x(t, 2\pi), & t\in (0,T),\\
	u(0,x)=u_0(x), v(0,x)=v_0(x), & x \in (0, 2\pi).
	\end{cases}
	\end{equation} 
	
	\noindent
	Its adjoint system is then given by:
	\begin{align}\label{adjoint}
	\begin{cases}
	-\varphi_t+\varphi_{xxxx}-\varphi_{xxx}+\varphi_{xx}+\psi_x=h_1, & (t,x)\in (0,T)\times (0,2\pi),\\
	-\psi_t-\psi_{xx}-\psi_x+\varphi_x=h_2,& (t,x)\in (0,T)\times (0,2\pi),\\
	\varphi(t,0)=\varphi(t,2\pi),\, \varphi_x(t,0)=\varphi_x(t,2\pi), & t \in (0,T),\\
	\varphi_{xx}(t,0)=\varphi_{xx}(t,2\pi),\, \varphi_{xxx}(t,0)=\varphi_{xxx}(t,2\pi), & t \in (0,T),\\
	\psi(t,0)=\psi(t,2\pi),\, \psi_x(t,0)=\psi_x(t,2\pi), & t \in (0,T),\\
	\varphi(T,x)=\varphi_T(x), \, \psi(T,x)=\psi_T(x), & x\in (0,2\pi).
	\end{cases}		
	\end{align}  
	\begin{proposition}\label{adj well-prop}
		Let us denote the space $\mathbf{X}=L^2(0,T;L^2(0,2\pi))^2 \text{ or }L^1(0,T;\mathbf{H})$. Let $(h_1,h_2)\in \mathbf{X}$ and $(
		\varphi_T , \psi_T
		) \in \mathbf{H}$. Then the system \eqref{adjoint} has a unique solution $(
		\varphi,\psi
		) \in C([0,T];\mathbf{H})$. Moreover, we have the following estimate:
		\begin{align}\label{enrgy_est 1}
		\left|\left|(\varphi, \psi)\right|\right|_{C([0,T];\mathbf{H})\cap L^2(0,T;H^4\times H^2)} \leq C\left(||(h_1,h_2)||_{\mathbf{X}}+\left|\left|(\varphi_T,\psi_T)\right|\right|_{\mathbf{H}}\right),
		\end{align}
		for some $C>0$ independent of $h_1, h_2, \varphi_T, \psi_T$.
	\end{proposition}
	\begin{proof}
		Proof of this proposition can be found in Appendix C (\Cref{sect7}).
	\end{proof}
	Using the continuous embedding of $H^k(0,2\pi)$ in $C^{k-1}(0,2\pi)$ for $k=2,4$, we conclude the follwoing from \Cref{adj well-prop}:
	\begin{proposition}\label{prop2.2}
		There exist $C>0$ such that the solution $(\varphi,\psi)$ of \eqref{adjoint} satisfies
		\begin{align}
			\nonumber ||\psi_x(\cdot,0)||_{L^2(0,T)}+||\varphi_{xxx}(\cdot,0)||_{L^2(0,T)}&+||\varphi_{xx}(\cdot,0)||_{L^2(0,T)}\\&\leq C\left(\,||(h_1,h_2)||_{\mathbf{X}}+\left|\left|(\varphi_T,\psi_T)\right|\right|_{\mathbf{H}}\right).
		\end{align}
	\end{proposition}

	Let us now define the solution of the system \eqref{eq:main1} by means of transposition scheme to state the well-posedness result for the same.
	\begin{definition}
		Let $(u_0,v_0)\in \mathbf{H}^*$,  $(g_1,g_2)\in L^2(0,T;(H^2_{per})\times L^2(0,T;(H^1_{per}))$ and $(q_1,q_2)\in (L^2(0,T))^2$. We say $(u,v)\in (L^2(0,T;L^2))^2$ is a solution of the system \eqref{eq:main1}, if for any $(h_1,h_2) \in (L^2(0,T;L^2))^2$ it satisfies:
		\begin{align*}
		\nonumber &\int_{0}^T \int_0^{2\pi} {u(t,x)}\overline {h_1(t,x)}\,dx\,dt+\int_0^T\int_0^{2\pi}  v(t,x)\overline{h_2(t,x)}\,dx\,dt\\&\nonumber=\int_0^T \left(-\overline{\varphi_{xxx}(t,0)}+\overline{\varphi_{xx}(t,0)}-\overline{\varphi_x(t,0)}-\overline{\psi(t,0)}\right)q_1(t)\, dt+\int_0^T\left(\overline{ \psi_x(t,0)}+\overline{ \psi(t,0)}-\overline{ \varphi(t,0)}\right)q_2(t)\,dt\\&\hspace{69 mm}+\ip{g_1}{\varphi}_{L^2((H^2_{per})^*),L^2(H^2_{per})}+\ip{g_2}{\psi}_{L^2((H^1_{per})^*),L^2(H^1_{per})}\\&\hspace{75 mm}+\ip{u_0}{\varphi(0,\cdot)}_{(H_{per}^2)^*,H_{per}^2}+\ip{v_0}{\psi(0,\cdot)}_{{(H_{per}^1)}^*, H_{per}^1},
		\end{align*}
		where $(\varphi, \psi
		)$ is solution of \eqref{adjoint} with $(\varphi(T,\cdot),\psi(T,\cdot))=(0,0)$.
	\end{definition}
	\begin{proposition}\label{well-prop3}
		Let $(g_1,g_2) \in L^2(0, T; {(H^2_{per})}^*) \times L^2(0, T; {(H^1_{per})}^*)$. Then, for any $(u_0,v_0)\in \mathbf{H}^*$ and $(q_1,q_2)\in (L^2(0,T))^2$, system \eqref{eq:cntrl2} has a unique solution $(u,v) \in C([0,T];\mathbf{H}^*)\cap L^2(0,T;L^2)^2$.
	\end{proposition}
	\begin{proof}
		Using \Cref{adj well-prop} and \Cref{prop2.2}, and following the proof of well-posedness result in \cite{EC12} (Theorem 2.8), one can easily obtain the above result.
	\end{proof}
	
	\section{Spectral analysis for the adjoint operator $A^*$}\label{sect3}

	The adjoint of the underlying spatial operator $A$ given by \ref{A}, can be written as:
	\begin{equation}\label{eq:A*}
	A^*=\begin{pmatrix} - \dfrac{d^4}{dx^4}+\dfrac{d^3}{dx^3}- \dfrac{d^2}{dx^2} & &-\dfrac{d}{dx}  \\ \\ -\dfrac{d}{dx}  &  &\dfrac{d^2}{dx^2}+\dfrac{d}{dx}  \end{pmatrix}  ,
	\end{equation}
	with $$D(A^*)=D(A)=H^4_{{per}}\tor\times {H}^2_{{per}}\tor.$$
	
	Thus, the adjoint system of \eqref{eq:main} given as $-\varPhi'(t)=A^*\varPhi(t)$ is equivalent to
	\begin{equation}\label{eq:adj}
	\begin{cases}
	\varphi_t- \varphi_{xxxx}+\varphi_{xxx} -  \varphi_{xx}- \psi_x=0, &  x \in \tor , \,t \in (0,T)\\
	\psi_t +\psi_{xx}+ \psi_x - \varphi_x=0, &  x \in \tor , \,t \in (0,T),\\
	\varphi(t,0)= \varphi(t, 2\pi), \varphi_x(t,0)= \varphi_x(t, 2\pi), & t \in (0,T), \\
	\varphi_{xx}(t,0)= \varphi_{xx}(t, 2\pi), \varphi_{xxx}(t,0)= \varphi_{xxx}(t, 2\pi), & t \in (0,T),\\
	\psi(t,0)= \psi(t, 2\pi),\,\psi_x(t,0)= \psi_x(t, 2\pi) & t \in (0,T),\\
	\varphi(T,x)= \varphi_T(x), \psi(T,x)= \psi_T(x), &  x \in \tor,
	\end{cases}
	\end{equation}
	where, $\varPhi=\begin{pmatrix}
	\varphi \\ \psi
	\end{pmatrix}$.
	
	\noindent
	We now compute the eigenvalues of $A^*$. Let us write the eigen equation corresponding to the adjoint operator $A^*$ as  
	\begin{equation*}
	A^*\begin{pmatrix}
	\xi\\
	\eta	
	\end{pmatrix}=\lambda \begin{pmatrix}
	\xi\\
	\eta	
	\end{pmatrix}, \quad \lambda \in \cplx,
	\end{equation*}
	which is explicitly the following
	\begin{equation*}
	\begin{cases}
	-\xi^{(iv)}+\xi'''-\xi''-\eta'=\lambda \xi,\\
	\eta''+\eta'-\xi'=\lambda \eta,\\
	\xi^i(0)=\xi^i(2\pi), \quad i=0,1,2,3,\\
	\eta^j(0)=\eta^j(2\pi), \quad j=0,1.
	\end{cases}
	\end{equation*} 
	Combining the first two equations, we get
	\begin{equation}\label{eq:eigen}
	-\xi^{(vi)}+\lambda \xi^{(iv)}-(1+\lambda)\xi'''-\xi''-\lambda \xi'+\lambda^2 \xi=0.
	\end{equation}
	Expanding $\xi$ as a Fourier series $\xi=\sum_{k\in \z} \xi_k e^{ikx}$, and substituting in \eqref{eq:eigen}, we get
	\begin{equation*}
	\lambda^2+\lambda (k^4+ik^3-ik)+k^6+ik^3+k^2=0, \text{ for each }k\in \z.
	\end{equation*}
	For $k=0,$ the above equation has double roots $\lambda_0=0,0$ and when $k\in \z \setminus\{0\}$,
	the expression
	\begin{equation*}
	\lambda_k^\pm=\frac{1}{2}[-(k^4+ik^3-ik)\pm\sqrt{(k^4+ik^3-ik)^2-4(k^6+ik^3+k^2)} ] 
	\end{equation*}
	solves the equation, which have the following asymptotic expressions:
	\begin{align}
\label{sph}	&\lambda_k^{+}=-k^2+ik+iO(|k|^{-1})+O(|k|^{-2}), \quad \text{as} \abs{k}\to \infty \\
\label{spp}	&\lambda_k^{-}=-k^4-ik^3+k^2+iO(|k|^{-1})+O(|k|^{-2}), \quad \text{as} \abs{k}\to \infty .
	\end{align}
	\begin{figure}[H]
	\centering
			\includegraphics[width=0.8\textwidth, height=0.5\textwidth]{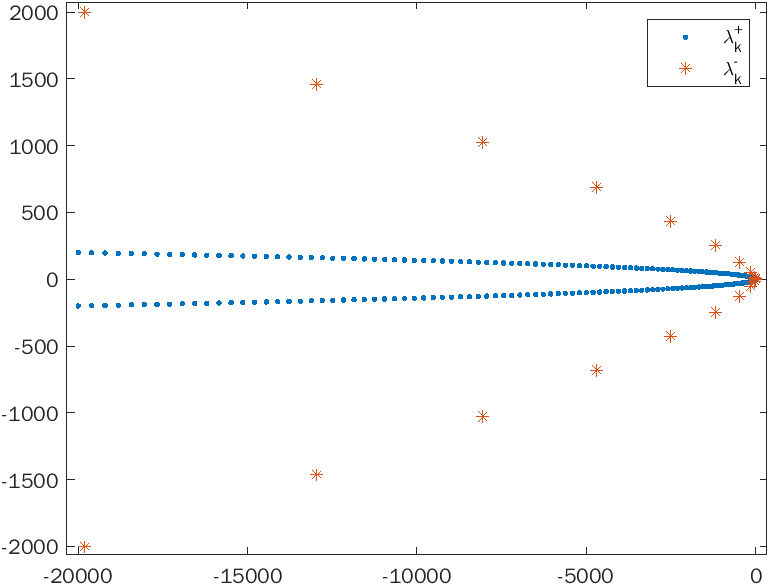}
			\caption{Spectral behavior of $A^*$}
			\label{fig1}
	\end{figure}
	\begin{figure}[H]
		\begin{subfigure}[b]{0.47\textwidth}
			\includegraphics[width=1\textwidth]{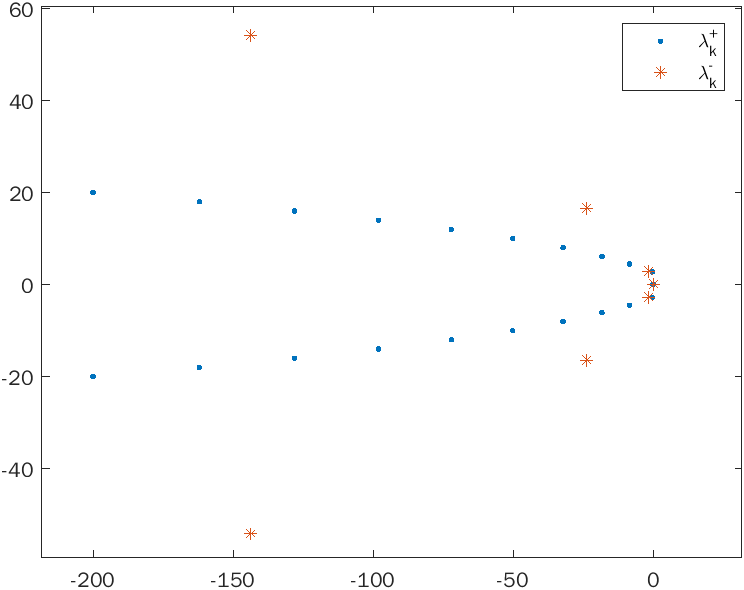}
		\end{subfigure}
		\begin{subfigure}[b]{0.47\textwidth}
			\includegraphics[width=1\textwidth]{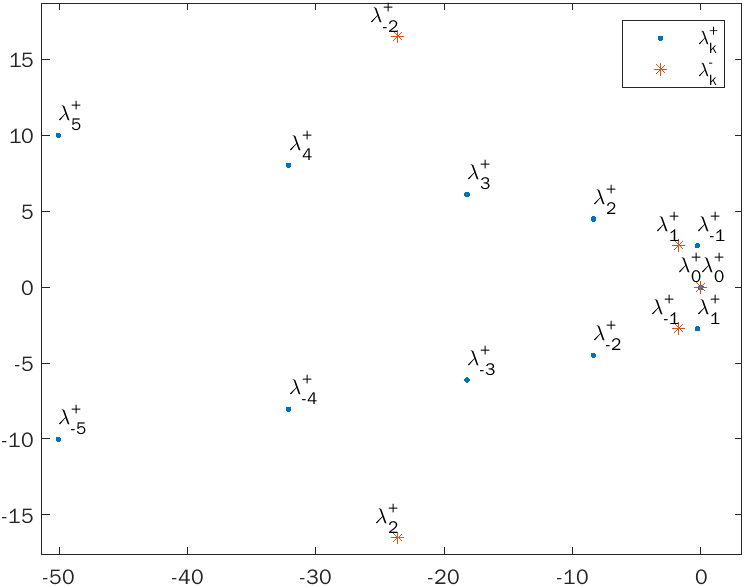}
		\end{subfigure}
		\caption{$\lambda_{k}^\pm$ are distinct for first finite number of $k(\neq0)$, symmetric about 0}
		\label{fig2}
	\end{figure}
	\begin{remark}\label{rem}
		Using the asymptotic expression of $\lambda_k^\pm$, \eqref{sph} and \eqref{spp}, and \Cref{fig2}, one can easily conclude that all the elements of the set of eigenvalues $\{\lambda_{k}^\pm\}_{k\in\z\setminus\{0\}}$ are distinct and nonzero. 
	\end{remark}
	Note that $\Phi_0=\begin{pmatrix}
	1\\0
	\end{pmatrix} \text{and}\, \hat{\Phi}_0=\begin{pmatrix}
	0\\1
	\end{pmatrix}$ are two linearly independent eigenvectors corresponding to $\lambda_0=0.$ For $k\in \z\setminus\{0\}$, the eigenvectors are given by 
	\begin{align*}
	\Phi_k^{+}=\begin{pmatrix}
	e^{ikx} \\\theta_k^{+}e^{ikx}\end{pmatrix}, \quad \Phi_k^{-}= \begin{pmatrix}
	e^{ikx} \\\theta_k^{-}e^{ikx}
	\end{pmatrix},
	\end{align*}
	where $\theta_k^\pm=\frac{\eta_k^\pm}{\lambda_k^\pm}$ with  $\eta_k^\pm=-ik^5-(1+\lambda_k^\pm)ik+k^2-\lambda_k^\pm$ are well defined for $k\neq0$ as $\lambda_k^\pm\neq0$  by \Cref{rem}. 
	\begin{figure}[H]
		\centering
		\begin{subfigure}[b]{0.32\textwidth}
			\includegraphics[width=1\textwidth]{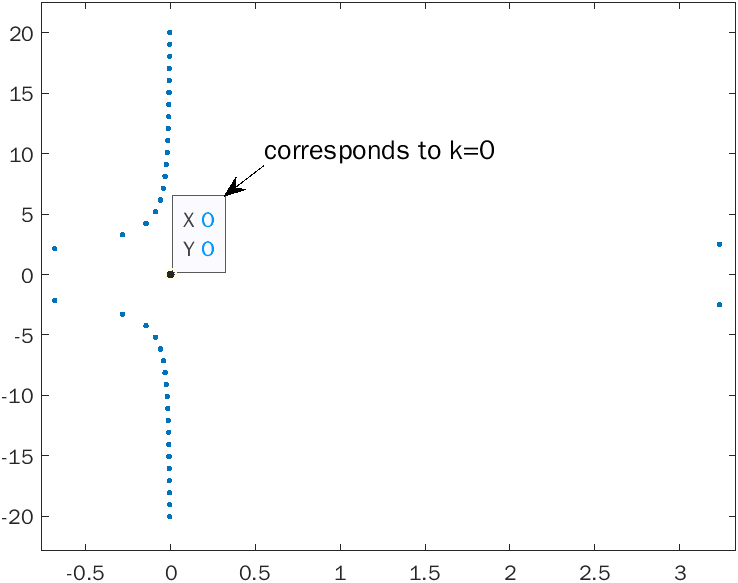}
			\caption{$\eta_k^+$}
		\end{subfigure}
		\begin{subfigure}[b]{0.32\textwidth}
			\includegraphics[width=1\textwidth]{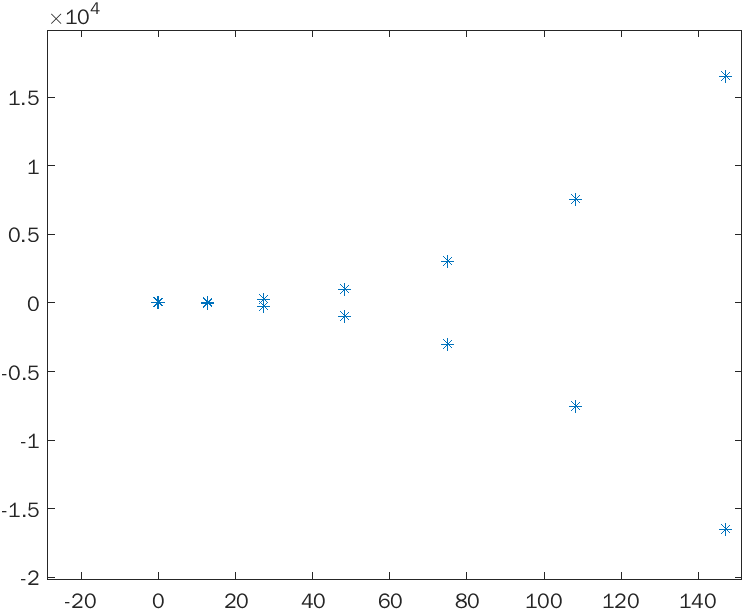}
			\caption{$\eta_k^-$}
		\end{subfigure}
		\begin{subfigure}[b]{0.32\textwidth}
			\includegraphics[width=1\textwidth]{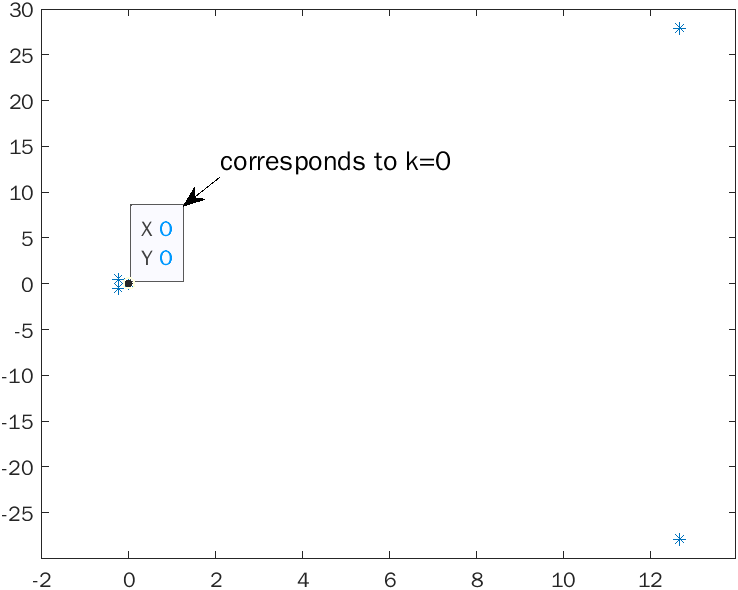}
			\caption{zoomed $\eta_k^-$}
		\end{subfigure}
		\caption{Plot of $\eta_k^\pm$ for first finite number of $k$, symmetric about 0}
		\label{fig3}
	\end{figure}
\begin{remark}\label{rem1}
	Using the asymptotic expression of $\eta_k^\pm$ and \Cref{fig3}, one can easily conclude that $\eta_k^\pm$ and hence $\theta_k^\pm$ are nonzero for $k\neq 0$.
	Moreover, we have:
	\begin{align*}
		\abs{\theta_k^+}<  C|k|^3 \text{ and } \abs{{\theta_k^-}} < \frac{c}{|k|^3},\text{ for } c,C>0, \text{ as } |k| \to \infty.
	\end{align*}
\end{remark}
	We now state the result corresponding to existence of biorthogonal family of $\{1, e^{\lambda_k^+t}, e^{\lambda_k^-t}\}_{k\in \z\setminus\{0\}}$, which will eventually help to solve the moment problems, to be derived in \Cref{sect4} and \Cref{sect5}.
	\begin{proposition}[\textbf{Biorthogonal family}]\label{prop1}
		Let $T>0$ and denote $\lambda_0=\lambda_0^+=\lambda_0^-=0.$ Then there exists a family $\{\Theta_k^{\pm}\}_{k\in \z\setminus\{0\}} \cup \{\Theta_0\}$ of functions in $L^2\left(-\frac{T}{2}, \frac{T}{2}\right)$ such that 
		\begin{equation*}
			\begin{cases}
		\int_{-\frac{T}{2}}^{\frac{T}{2}} \Theta_k^{\pm}(t)  e^{-\overline{\lambda_l^\pm} t} dt= \delta_{kl} \delta_\pm,\hspace{2 mm} l\in \zahl, k \in \z\setminus\{0\}, \vspace{0.2cm}\\ 
		\int_{-\frac{T}{2}}^{\frac{T}{2}} \Theta_0(t)  e^{\overline{-\lambda_l^\pm} t} dt=\delta_{0l},\hspace{2 mm} l \in \z,
		\end{cases}
		\end{equation*}	
		where,
		\begin{align*}
		\delta_{kl}=\begin{cases}1, & \text{ if } k=l, \\
		0, & \text{ if } k\neq l\end{cases} 
		\text{ and } 
		\delta_{\pm}=\begin{cases}1, & \text{ if sign on $\Theta_k$ and $\lambda_k$ is same}  \\
		0, & \text{ otherwise }.\end{cases} 
		\end{align*}
		Moreover, we have the following estimates
		\begin{align*}
		||\Theta_k^+||_{L^2(\rea)}&\leq C\,|k|^2\,e^{-\frac{T}{2} k^2 -2\pi |k|^{\frac{1}{2}}+3(3+2\sqrt{2})\pi|k|},\\
		||\Theta_k^-||_{L^2(\rea)}&\leq C |k|^5\, e^{-\frac{T}{2} k^4 +(2\sqrt{2}+1)\pi |k|^2+c\pi|k|},\\
		||\Theta_0||_{L^2(\rea)} &\leq C,
		\end{align*}
	where $C, c$ are two positive constants independent of $k.$
	\end{proposition}
	
	\section{Interior null controllability: The Method of Moment}\label{sect4}
	\noindent
	This section is devoted to the proof of \Cref{thm1} and \Cref{thm1a}.
	\subsection{Reduction to the moment problem} In this section, we deduce a system of identities, called moment problem, corresponding to systems \eqref{eq:cntrl1} and \eqref{eq:cntrl1a}. We also show that proving \Cref{thm1} and \Cref{thm1a} is equivalent to solving the moment problems corresponding to the considered control systems. 
	\subsubsection{\textbf{Interior bilinear distributed control acting in the KS-KdV equation}}\label{sect4.1.1}
		Let us first assume the initial data $(u_0, v_0)$ of \eqref{eq:cntrl1}, the terminal data $(\varphi_T, \psi_T)$ of \eqref{eq:adj}, $f$ and $g$  to be smooth enough.
		This eventually ensures the solution $(u,v)$ and $(\varphi,\psi)$ of systems \eqref{eq:cntrl1} and \eqref{eq:adj}, respectively to be smooth. We then take inner product of the first two equations of \eqref{eq:cntrl1} with $\varphi$ and $\psi$, respectively to get
		
		\begin{equation}\label{eq:ident}
		\ip{\begin{pmatrix}
			u(T,\cdot)\\ v(T, \cdot)
			\end{pmatrix}}{\begin{pmatrix}
			\varphi_T\\ \psi_T
			\end{pmatrix}}_{\mathbf{H^*},\mathbf{H}}-\ip{\begin{pmatrix}
			u_0\\ v_0
			\end{pmatrix}}{\begin{pmatrix}
			\varphi(0, \cdot)\\ \psi(0, \cdot)
			\end{pmatrix}}_{\mathbf{H^*},\mathbf{H}}
		=\int_{0}^{T}\intg f(x)g(t)\overline{\varphi(t,x)} \,dx dt.
		\end{equation}
		Using density argument, one can show the above identity to be true even if $f\in \lt, g \in L^2(0,T)$ and the initial data $(u_0, v_0)$, terminal data $(\varphi_T,\psi_T)$ lie in $\mathbf{H^*}$ and $\mathbf{H}$, respectively.
		
		\vspace{1 mm}
		\noindent
		Assume the system \eqref{eq:cntrl1} is null controllable, then \eqref{eq:ident} implies
		\begin{align}\label{int ident}
		-\ip{\begin{pmatrix}
			u_0\\ v_0
			\end{pmatrix}}{\begin{pmatrix}
			\varphi(0, \cdot)\\ \psi(0, \cdot)
			\end{pmatrix}}_{\mathbf{H^*},\mathbf{H}}
		=\int_{0}^{T}\intg f(x)g(t)\overline{\varphi(t,x)}\, dx dt.
		\end{align}
		\noindent
		Conversely, assume that \eqref{int ident} holds then \eqref{eq:ident} gives
		\begin{align*}
		\ip{\begin{pmatrix}
			u(T,\cdot)\\ v(T, \cdot)
			\end{pmatrix}}{\begin{pmatrix}
			\varphi_T\\ \psi_T
			\end{pmatrix}}_{\mathbf{H^*},\mathbf{H}}=0,
		\end{align*}
		where $
		(\varphi_T,\psi_T )
	 \in \mathbf{H}$ and hence we get
		\begin{align*}
		\begin{pmatrix}
		u(T,\cdot)\\
		v(T,\cdot)
		\end{pmatrix}=\begin{pmatrix}
		0\\0
		\end{pmatrix} \text{ in $\mathbf{H^*}$}.
		\end{align*}
		
		\noindent
		Proving  \Cref{thm1} is equivalent to finding $f\in L^2(0,2\pi)$, $g \in L^2(0,T)$ such that $g$ satisfies \eqref{int ident} with $(\varphi_T,\psi_T )$ varying over $\mathbf{H}$, which is equivalent to varying $(\varphi_T,\psi_T )$ over some basis of $\mathbf{H}$. Note that the set of eigenfunctions of $A^*$ forms basis of $\mathbf{H}$ and so we will consider these elements as terminal condition $(\varphi_T,\psi_T )$ and then substitute it in the identity \eqref{int ident} to get a system of identities, called moment problem, which will be solved in \Cref{sect4.2} to prove \Cref{thm1}.
		
		Let us first take the terminal data $(\varphi_T, \psi_T)$ as
		\begin{align*}	\begin{pmatrix}
		\varphi_T\\ \psi_T
		\end{pmatrix}= \begin{pmatrix}
		e^{ikx} \\\theta_k^{\pm}e^{ikx}\end{pmatrix}, \quad k \in \z\setminus \{0\}.
		\end{align*}
		Then the solution of the adjoint system \eqref{eq:adj} is of the form
		\begin{align*}
		\begin{pmatrix}
		\varphi(t,x)\\ \psi(t,x)
		\end{pmatrix}= e^{\lambda_k^\pm(T-t)}\begin{pmatrix}
		e^{ikx} \\\theta_k^{\pm}e^{ikx}\end{pmatrix}, \quad k \in \z\setminus \{0\}	
		\end{align*}
		On substituting it in the identity \eqref{int ident}, we get 
		\begin{align}\label{eq:moment1}
		\nonumber &e^{\overline{\lambda_k^\pm}\frac{T}{2}}	\left(\ip{u_0}{e^{ikx}}_{(H^{2}_{per})^*,H^2_{per}}+ \overline{\theta_k^\pm} \ip{v_0}{e^{ikx}}_{(H^{1}_{per})^*,H^1_{per}}\right) +f_k\int_{-\frac{T}{2}}^{\frac{T}{2}}e^{-\overline{\lambda_k^\pm}t} g\left(t+\frac{T}{2}\right)=0,\\
		&i.e.,\hspace{20 mm} f_k\int_{-\frac{T}{2}}^{\frac{T}{2}}e^{-\overline{\lambda_k^\pm}t}\, g\left(t+\frac{T}{2}\right)=-e^{\overline{\lambda_k^\pm}\frac{T}{2}}\gamma_k^\pm,
		\end{align}
		where, $f_k=\intg f(x)e^{-ikx}$ and $\gamma_k^\pm = \ip{u_0}{e^{ikx}}_{(H^{2}_{per})^*,H^2_{per}}+ \overline{\theta_k^\pm} \ip{v_0}{e^{ikx}}_{(H^{1}_{per})^*,H^1_{per}} \text{ for } k \in \z\setminus \{0\}$.
		
		\noindent
		Next we take the terminal data as \begin{align*}	\begin{pmatrix}
		\varphi_T\\ \psi_T
		\end{pmatrix}= \Phi_0=\begin{pmatrix}
		1\\0
		\end{pmatrix} .
		\end{align*}
		Then, identity \eqref{int ident} gives 
		\begin{align}\label{eq:moment2}
		f_0\int_{-\frac{T}{2}}^{\frac{T}{2}}g\left(t+\frac{T}{2}\right) dt=-\gamma_0,
		\end{align}
		where $f_0=\intg f(x)\,dx$ and $\gamma_0=\ip{u_0}{1}_{(H^{2}_{per})^*,H^2_{per}}$.
		
		Finally taking the terminal data as \begin{align*}	\begin{pmatrix}
		\varphi_T\\ \psi_T
		\end{pmatrix}= \hat{\Phi}_0=\begin{pmatrix}
		0\\1
		\end{pmatrix} ,
		\end{align*}
		the identity \eqref{int ident} gives 
		\begin{align*}
		\ip{v_0}{1}_{(H^{1}_{per})^*,H^1_{per}}=0.
		\end{align*} 
		Thus, we get the following lemma, which is equivalent to \Cref{thm1}:
		\begin{lemma}\label{lem4.1}
			The control system \eqref{eq:cntrl1} is null controllable if and only if for any initial data $
			(u_0,v_0)
			 \in \mathbf{H}^*$ with $\ip{v_0}{1}_{{(H^1_{per})}^*,H^1_{per}}=0$, there exist functions $f\in L^2\tor, g\in L^2(0,T)$ such that the following hold
			\begin{align}\label{intmoment1}
				\begin{cases}
					f_k\int_{-\frac{T}{2}}^{\frac{T}{2}}e^{-\overline{\lambda_k^\pm}t}g\left(t+\frac{T}{2}\right)=-e^{\overline{\lambda_k^\pm}\frac{T}{2}}\gamma_k^\pm, \vspace{0.2cm} \\
					f_0\int_{-\frac{T}{2}}^{\frac{T}{2}}g(t)dt=-\gamma_0,
				\end{cases}
			\end{align}
			where \begin{align}\label{constant}
			\nonumber	f_k=\intg f(x)e^{-ikx}\,dx, &\quad \gamma_k^\pm = \ip{u_0}{e^{ikx}}_{{(H^2_{per})}^*,H^2_{per}}+ \overline{\theta_k^\pm} \ip{v_0}{e^{ikx}}_{{(H^1_{per})}^*,H^1_{per}},\,  \forall k \in \z\setminus \{0\} ,\\
				& f_0=\intg{f(x) \,dx} \text{ and } \gamma_0=\ip{u_0}{1}_{{(H^2_{per})}^*,H^2_{per}}.
			\end{align}
		\end{lemma}
	\subsubsection{\textbf{Interior bilinear distributed control acting in the heat equation}}In this section, we consider the control system \eqref{eq:cntrl1a} with localized interior control on the heat component.
	Employing the argument of the last section, one can easily obtain the following analogous lemma, which is equivalent to \Cref{thm1a}:
	\begin{lemma}
		The control system \eqref{eq:cntrl1a} is null controllable if and only if for any initial data $(u_0,v_0)
	 \in \mathbf{H}^*$ with $\ip{u_0}{1}_{{(H^2_{per})}^*,H^2_{per}}=0$, there exist control functions $f\in L^2\tor, g\in L^2(0,T)$ solving the following moment problem:
		\begin{align}\label{intmoment2}
			\begin{cases}
				\overline{\theta_k^\pm}f_k\int_{-\frac{T}{2}}^{\frac{T}{2}}e^{-\overline{\lambda_k^\pm}t}g\left(t+\frac{T}{2}\right)=-e^{\overline{\lambda_k^\pm}\frac{T}{2}}\gamma_k^\pm, \vspace{0.2cm} \\
				f_0\int_{-\frac{T}{2}}^{\frac{T}{2}}g(t)dt=-\gamma_0,
			\end{cases}
		\end{align}
		where $f_k, \gamma_k^\pm, f_0 $ are same as \eqref{constant} and $ \gamma_0=\ip{v_0}{1}_{{(H_{per}^1)}^*,H_{per}^1}.$
	\end{lemma}
\subsection{Proof of \Cref{thm1}}\label{sect4.2}
		\begin{proof}
			As argued in \Cref{sect4.1.1}, proving this theorem is equivalent to solving the moment problem \eqref{intmoment1}.
			Let us define $g$ formally as  
			\begin{equation}\label{g}
			g(t)=-\sum_{k\in \z\setminus\{0\}}e^{\overline{\lambda_k^+}\frac{T}{2}}f_k^{-1}\gamma_k^+\,\Theta_k^+\left(t-\frac{T}{2}\right)-\sum_{k\in \z\setminus\{0\}}e^{\overline{\lambda_k^-}\frac{T}{2}}f_k^{-1}\gamma_k^-\,\Theta_k^-\left(t-\frac{T}{2}\right)-f_0^{-1}\gamma_0 \, \Theta_0\left(t-\frac{T}{2}\right).
			\end{equation}
			
			\noindent
			For the above definition to be sensible, one needs to make sure that $f_k \neq 0, \, \forall\, k \in \zahl$. Existence of such $f$ can be ensured from the lemma given below.
			
			\begin{lemma}
				Let $\omega$ be any nonempty subset of $(0,2\pi)$ and let $\alpha \in \omega$ and $\rho \in (0,1)$ be a quadratic irrational (irrational number which is a root of quadratic equation with integral coefficients) such that $[\alpha,\alpha+\rho\pi]$ is subset of $\omega$. Define
				\begin{align*}
				f(x)=\chi_{[\alpha,\alpha+\rho\pi]}(x),\quad \forall x \in (0,2\pi).
				\end{align*}
				Then, $f \in L^2(0,2\pi)$ with support inside $\omega$ and $f_k=\intg f(x) e^{-ikx}\,dx \neq 0$, for $k\in\z$.
				Moreover, there exist $\widetilde{C}>0$ such that $|f_k|\geq\frac{C}{|k|^2}$, for all $ k \in\z\setminus\{0\}$.
			\end{lemma}
			\begin{proof}
				
				Clearly, $f \in L^2(0,2\pi)$ with support inside $\omega$ and also  
				\begin{align*}
				&\quad\quad\quad\quad\quad\quad\quad f_0=\rho \pi \neq 0\\
				&f_k=\intg f(x)e^{-ikx} = \frac{e^{-ik\alpha}}{ik}(1-e^{-ik\rho \pi}) \neq 0,\,\, \forall\, k\in\z\setminus\{0\}
				\end{align*}
				Now, as $\rho$ is quadratic irrational so it can be approximated by rational numbers to order 2 and to no higher order (\cite{bell1939gh}, Theorem 188), i.e., there exist $C>0$ such that for any integers $p$ and $q$, $q \neq 0$,
				\begin{align}\label{quadirr}
				\left|\rho-\frac{p}{q}\right|\geq \frac{C}{q^2}. 
				\end{align}
				\noindent
				Also, for $k \in \z\setminus\{0\}$
				\begin{align}\label{fk}
			\nonumber	|f_k|=\frac{1}{|k|}|1-e^{-ik\rho\pi}| &= \frac{1}{|k|}\left|\,2\, \sin^2\left(\frac{k\rho\pi}{2}\right)+i\,2\, \sin\left(\frac{k\rho\pi}{2}\right)\, \cos\left(\frac{k\rho\pi}{2}\right)\right|\\
				& =2\frac{\left|\sin\left(\frac{k\rho\pi}{2}\right)\right|}{|k|}.
				\end{align}
				Note that,
				\begin{align}\label{sin}
					\sin^2x\geq \frac{4x^2}{\pi^2} \text{, for }x\in \left[-\frac{\pi}{2},\frac{\pi}{2}\right].
				\end{align}
				For any fixed $k\in\z\setminus\{0\}$, choose $p \in \z$ such that $0\leq \frac{k\rho\pi}{2} - p\pi\leq\pi$. 
				\begin{enumerate}
					\item[Case I. ] If $0\leq\frac{k\rho\pi}{2} - p\pi\leq\frac{\pi}{2}$,  then  we have
					\begin{align*}
						\sin^2\left(\frac{k\rho\pi}{2}\right)=\sin^2\left(\frac{k\rho\pi}{2} - p\pi\right)&\geq k^2\left(\rho-\frac{2p}{k}\right)^2, \text{ by \eqref{sin}} \\
						&\geq k^2 \frac{C}{k^4} =  \frac{C}{k^2}\text{ by \eqref{quadirr}}.
					\end{align*}
					\item[Case II. ] If $\frac{\pi}{2}\leq\frac{k\rho\pi}{2} - p\pi\leq\pi$, i.e., $-\frac{\pi}{2}\leq\frac{k\rho\pi}{2} - (p+1)\pi\leq0$, then we have
						\begin{align*}
						\sin^2\left(\frac{k\rho\pi}{2}\right)=\sin^2\left(\frac{k\rho\pi}{2} - (p+1)\pi\right)\geq  \frac{C}{k^2}, \text{ by last case}.
					\end{align*}
				\end{enumerate}
				Combining the above two cases, we have
				\begin{align}\label{fkone}
				\sin^2\left(\frac{\pi}{2}k\rho\right)\geq \frac{C^2}{k^2},\,\, \forall\, k\in\z\setminus\{0\}.
				\end{align}
			 Plugging the estimate \eqref{fkone} in \eqref{fk}, we obtain $\left|f_k\right| \geq \frac{\widetilde{C}}{|k|^2}$ for some $\widetilde{C}>0$, $k \in \z \setminus\{0\}.$
				\end{proof}
			Thus, the above lemma justifies the expression \eqref{g} of $g$ and also from proposition \eqref{prop1}, one can easily conclude that $g$ solves the moment problem \eqref{intmoment1} with the restriction $\ip {v_0}{1}_{{(H^2_{per})}^*, H^2_{per}}=0$.
			
			\noindent
			 Now, we only need to show that $g \in L^2(0,T)$.
			\begin{align*}
			&||g||_{L^2(0,T)}\\&\hspace{3mm}\leq \sum_{k\in \z\setminus\{0\}}\left|e^{\overline{\lambda_k^+}\frac{T}{2}}\right|\,|f_k|^{-1}\,|\gamma_k^+|\,||\Theta^+_k||_{L^2(-\frac{T}{2},\frac{T}{2})}+\sum_{k\in \z\setminus\{0\}}\left|e^{\overline{\lambda_k^-}\frac{T}{2}}\right|\,|f_k|^{-1}\,|\gamma_k^-|\,||\Theta^-_k||_{L^2(-\frac{T}{2},\frac{T}{2})}\\ & \hspace{80 mm}+|f_0|^{-1}\,|\gamma_0|\,||\Theta_0||_{L^2(-\frac{T}{2},\frac{T}{2})}\\
			&\hspace{3mm}\leq C\,\left( \sum_{k\in \z\setminus\{0\}}e^{-|k|^2\frac{T}{2}}|f_k|^{-1}\,\left(||u_0||_{{(H^{2}_{per})}^*}||e^{ikx}||_{H^2_{per}}+\left|\overline{\theta_k^+}\right| ||v_0||_{{(H^{1}_{per})}^*}||e^{ikx}||_{H^{1}_{per}}\right)||\Theta^+_k||_{L^2(-\frac{T}{2},\frac{T}{2})}\right.\\
			&\hspace{4 mm} \left. +\sum_{k\in \z\setminus\{0\}}e^{(-k^4+k^2)\frac{T}{2}}|f_k|^{-1}\,\left(||u_0||_{{(H^{2}_{per})}^*}||e^{ikx}||_{H^2_{per}}+\left|\overline{\theta_k^-}\right| ||v_0||_{{(H^{1}_{per})}^*}||e^{ikx}||_{H_{per}^{1}}\right)||\Theta^+_k||_{L^2(-\frac{T}{2},\frac{T}{2})}\right)\\
			&\hspace{3mm}\leq C\,\left( \sum_{k\in \z\setminus\{0\}}|k|^2\,\left(||u_0||_{(H^{2}_{per})^*}+||v_0||_{{(H^{1}_{per})}^*}\right)\,|k|^4|k|^2\,e^{-T |k|^2 -2\pi |k|^{\frac{1}{2}}+3(3+2\sqrt{2})\pi|k|}\right.\\
			&\hspace{20 mm} \left. +\sum_{k\in \z\setminus\{0\}}|k|^2\,\left(||u_0||_{{(H^{2}_{per})}^*}+||v_0||_{{(H^{1}_{per})}^*}\right)|k|^2\,|k|^5\, e^{-T |k|^4 +(2\sqrt{2}+1)\pi |k|^2+\frac{T}{2}|k|^2+c\pi|k|}\right)\\
			&\hspace{3mm}\leq C\left(\sum_{k\in \z\setminus\{0\}}|k|^8e^{-T |k|^2 -2\pi |k|^{\frac{1}{2}}+3(3+2\sqrt{2})\pi|k|}+\sum_{k\in \z\setminus\{0\}}|k|^{9}e^{-T |k|^4 +(2\sqrt{2}+1)\pi |k|^2+\frac{T}{2}|k|^2+c\pi|k|}\right)\\
			&\hspace{3mm}< \infty.
			\end{align*}
		\end{proof}
	\subsection{Proof of \Cref{thm1a}}
		As mentioned before, proving \Cref{thm1a} is equivalent to solving the moment problem \eqref{intmoment2}.
		Recall that $\theta_k^\pm\neq 0$(see \cref{rem1}), and so define $g$ as:
		\begin{align*}
			\nonumber g(t)=-\sum_{k\in \z\setminus\{0\}}e^{\overline{\lambda_k^+}\frac{T}{2}}(\theta_k^+)^{-1}f_k^{-1}\gamma_k^+\,\Theta_k^+\left(t-\frac{T}{2}\right)-&\sum_{k\in \z\setminus\{0\}}e^{\overline{\lambda_k^-}\frac{T}{2}}(\theta_k^-)^{-1}f_k^{-1}\gamma_k^-\,\Theta_k^-\left(t-\frac{T}{2}\right)\\&\quad\quad\quad\quad\quad\quad-f_0^{-1}\gamma_0 \, \Theta_0\left(t-\frac{T}{2}\right).
		\end{align*}
		Following the above proof of \Cref{thm1}, one can easily conclude that this $g\in L^2(0,T)$ solves the moment problem and hence \Cref{thm1a} follows.
	\section{Boundary null controllability: The Method of Moment}\label{sect5}
	\noindent
	In this section, we give the proofs of \Cref{thm2} and \Cref{thm2a}. We follow the same steps as in the case of interior controllability, i.e., we first derive the moment problem corresponding to each theorem and then solve it.
	\subsection{Reduction to moment problem} This section is devoted to derive the moment problem corresponding to the control systems \eqref{eq:cntrl2} and \eqref{eq:cntrl2a}. 
	\subsubsection{\textbf{Control acting on the KS-KdV component}}\label{sect5.1.1}
	 \noindent
	 Let us first assume that, the control $q$, initial data $(u_0, v_0)$ of \eqref{eq:cntrl2} and the terminal data $(\varphi_T, \psi_T )$ of \eqref{eq:adj} are smooth enough. 
	 Let us take the duality product $\ip{\cdot}{\cdot}_{{{\mathbf{H^*},\mathbf{H}}}}$ of the first two equations of \eqref{eq:cntrl2} with $\varphi$ and $\psi$, respectively and then perform integration by parts successively to get
	 \begin{align*}
	 \nonumber&\ip{\begin{pmatrix}
	 	u(T,\cdot)\\ v(T, \cdot)
	 	\end{pmatrix}}{\begin{pmatrix}
	 	\varphi_T\\ \psi_T
	 	\end{pmatrix}}_{{\mathbf{H^*},\mathbf{H}}}-\ip{\begin{pmatrix}
	 	u_0\\ v_0
	 	\end{pmatrix}}{\begin{pmatrix}
	 	\varphi(0, \cdot)\\ \psi(0, \cdot)
	 	\end{pmatrix}}_{\mathbf{H^*},\mathbf{H}}\\
	 	&\hspace{15 mm}=\int_{0}^{T}\left(\overline{\varphi_{xxx}(t, 2\pi)}-\overline{\varphi_{xx}(t,2\pi)}+\overline{\varphi_{x}(t,2\pi)}+\overline{\psi(t,2\pi)}\right)q(t) dt.
	 \end{align*}
	 By the density argument, this identity holds even if we take the initial data $
	 (u_0,v_0)$ and terminal data $(
	 \varphi_T, \psi_T )$ in $\mathbf{H^*}$ and $\mathbf{H}$, respectively.

	 As done in \Cref{sect4.1.1}, we get the following equivalent identity for null controllability of the control system \eqref{eq:cntrl2}: \begin{align}\label{eq:identity1}
	 \ip{\begin{pmatrix}
	 	u_0\\ v_0
	 	\end{pmatrix}}{\begin{pmatrix}
	 	\varphi(0, \cdot)\\ \psi(0, \cdot)
	 	\end{pmatrix}}_{\mathbf{H^*},\mathbf{H}}
	 +\int_{0}^{T}\left(\overline{\varphi_{xxx}(t, 2\pi)}-\overline{\varphi_{xx}(t,2\pi)}+\overline{\varphi_{x}(t,2\pi)}+\overline{\psi(t,2\pi)}\right)q(t) dt=0.
	 \end{align}
	 Now, we deduce the moment problem by varying the terminal data $(\varphi_T,\psi_T)$ over the set of eigenvectors of $A^*$, which forms a  basis of $\mathbf{H}$. Let us first take the terminal data $(\varphi_T, \psi_T)$ as
	 \begin{align*}	\begin{pmatrix}
	 \varphi_T\\ \psi_T
	 \end{pmatrix}= \begin{pmatrix}
	 e^{ikx} \\\theta_k^{\pm}e^{ikx}\end{pmatrix}, \quad k \in \z\setminus \{0\}.
	 \end{align*}
	 then the solution of the adjoint system \eqref{eq:adj} is given by
	 \begin{align*}
	 \begin{pmatrix}
	 \varphi(t,x)\\ \psi(t,x)
	 \end{pmatrix}= e^{\lambda_k^\pm(T-t)}\begin{pmatrix}
	 e^{ikx} \\\theta_k^{\pm}e^{ikx}\end{pmatrix}, \quad k \in \z\setminus \{0\},	
	 \end{align*}
	 and so we have:
	 \begin{align}\label{eq:phizero1}
	 \begin{pmatrix}
	 \varphi(0,x)\\ \psi(0,x)
	 \end{pmatrix}= e^{\lambda_k^\pm T}\begin{pmatrix}
	 e^{ikx} \\\theta_k^{\pm}e^{ikx}\end{pmatrix}, \quad k \in \z\setminus \{0\}.	
	 \end{align}
	 \begin{align}
	 \label{eq:bstarphi1}&\varphi_{xxx}(t,2\pi)=-ik^3e^{\lambda_k^\pm(T-t)}, \quad \varphi_{xx}(t,2\pi)=-k^2e^{\lambda_k^\pm(T-t)},\\
	 &\varphi_{x}(t,2\pi)=ike^{\lambda_k^\pm(T-t)}, \quad \psi(t,2\pi)=\theta_k^\pm e^{\lambda_k^\pm(T-t)}\label{eq:bstarphi2}.
	 \end{align}
	 Using \eqref{eq:phizero1}, \eqref{eq:bstarphi1}, \eqref{eq:bstarphi2} in \eqref{eq:identity1}, we get
	 \begin{align*}
	 e^{\overline{\lambda_k^\pm} T}\left(\ip{u_0}{e^{ikx}}_{(H_{per}^{2})^*,H_{per}^2}+ \overline{\theta_k^\pm} \ip{v_0}{e^{ikx}}_{(H_{per}^{1})^*,H_{per}^1}\right)+\left(ik^3+k^2-ik+\overline{\theta_k^\pm}\right)\int_{0}^{T}e^{\overline{\lambda_k^\pm}(T-t)} q(t)=0,
	 \end{align*}
	 which on simplification gives
	 \begin{align*}
	 &e^{\overline{\lambda_k^\pm}T}	\left(\ip{u_0}{e^{ikx}}_{(H_{per}^{2})^*,H_{per}^2}+ \overline{\theta_k^\pm} \ip{v_0}{e^{ikx}}_{(H_{per}^{1})^*,H_{per}^1}\right)\\&\hspace{30mm}+\left(ik^3+k^2-ik+\overline{\theta_k^\pm}\right)e^{\overline{\lambda_k^\pm}\frac{T}{2}}\int_{-\frac{T}{2}}^{\frac{T}{2}}e^{-\overline{\lambda_k^\pm}t} q\left(t+\frac{T}{2}\right)=0.\\
	 &i.e.,\hspace{10 mm}\quad \int_{-\frac{T}{2}}^{\frac{T}{2}}e^{-\overline{\lambda_k^\pm}t} q\left(t+\frac{T}{2}\right)=-e^{\overline{\lambda_k^\pm}\frac{T}{2}}\gamma_k^\pm,
	 \end{align*}
	 where \begin{equation*}\label{eq:gamma1}\gamma_k^\pm=\frac{1}{\left(ik^3+k^2-ik+\overline{\theta_k^\pm}\,\right)} \left(\ip{u_0}{e^{ikx}}_{(H_{per}^{2})^*,H_{per}^2}+ \overline{\theta_k^\pm} \ip{v_0}{e^{ikx}}_{(H_{per}^{1})^*,H_{per}^1}\right), \, k\in \zahl\setminus\{0\}\end{equation*} 
	 is well defined as from \Cref{fig4}, \Cref{fig5} given below, it is clear that $(ik^3+k^2-ik+\overline{\theta_k^\pm}\,)\neq 0$.
	 \begin{figure}[H] 
	 \centering 
	 \begin{subfigure}[b]{0.47\textwidth}
	 	\includegraphics[width=1\textwidth]{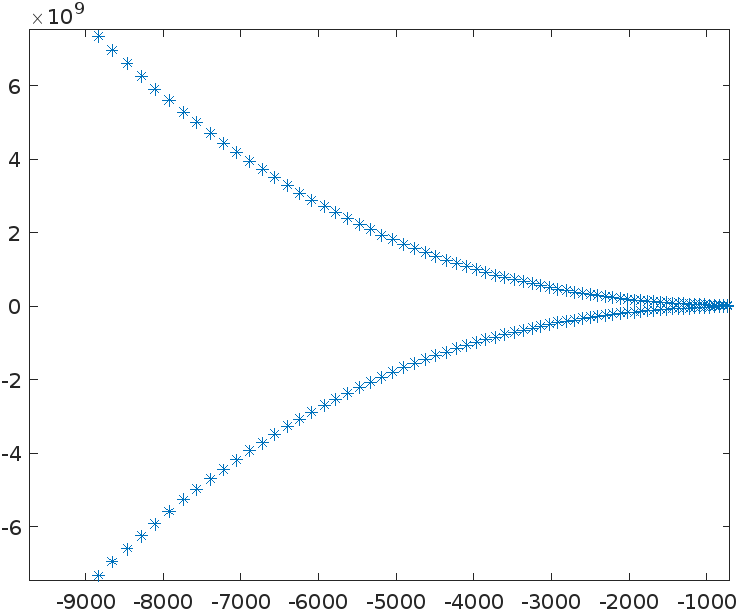}
	 	\caption{Assymptotic behaviour}
	 \end{subfigure}
	 \begin{subfigure}[b]{0.5\textwidth}
	 	\includegraphics[width=1\textwidth]{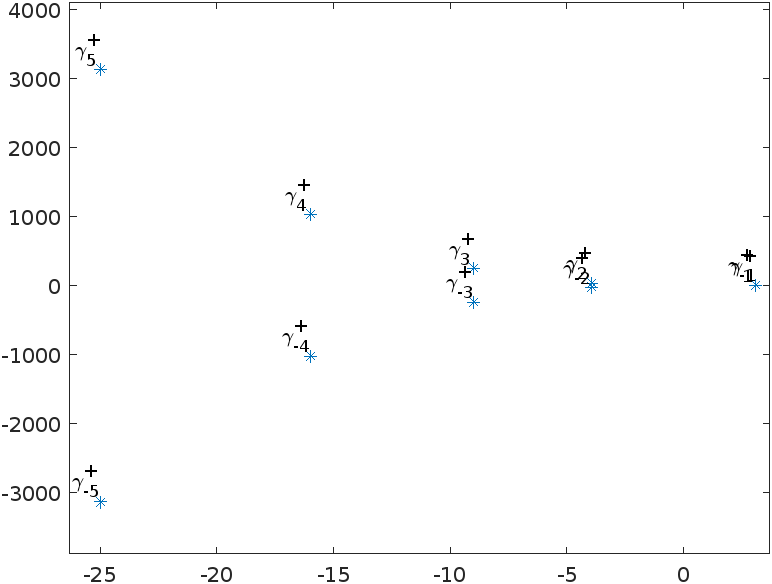}
	 	\caption{for $k\in[-5,5]\cap\z\setminus\{0\}$}
	 \end{subfigure}
 	\caption{Plot of $\lambda_k^+(-ik^3+k^2+ik)+\eta_k^+$ to show $(ik^3+k^2-ik+\overline{\theta_k^+})\neq0$}
 	\label{fig4}
	\end{figure}
	\begin{figure}[H]
		\centering
		\begin{subfigure}[b]{0.45\textwidth}
			\includegraphics[width=1\textwidth]{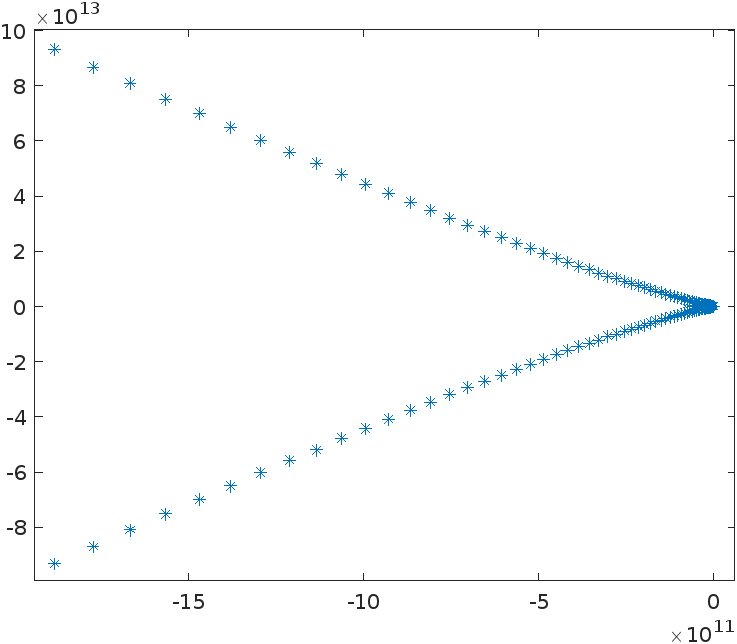}
			\caption{Assymptotic behaviour}
		\end{subfigure}
		\begin{subfigure}[b]{0.49\textwidth}
			\includegraphics[width=1\textwidth]{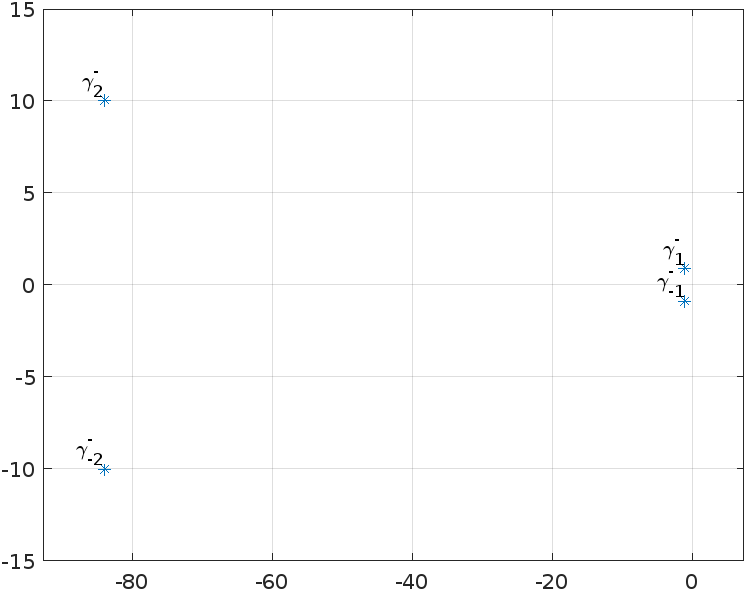}
			\caption{for $k\in[-2,2]\cap\z\setminus\{0\}$}
		\end{subfigure}
		\caption{Plot of $\lambda_k^-(-ik^3+k^2+ik)+\eta_k^-$ to show $(ik^3+k^2-ik+\overline{\theta_k^-})\neq0$}
		\label{fig5}.
	\end{figure}
Next we take the terminal data as \begin{align*}	\begin{pmatrix}
	 \varphi_T\\ \psi_T
	 \end{pmatrix}= \Phi_0=\begin{pmatrix}
	 1\\0
	 \end{pmatrix}.
	 \end{align*}
	 Then the identity \eqref{eq:identity1} gives \begin{align*}\ip{u_0}{1}_{(H_{per}^{2})^*,H_{per}^2}=0.
	 \end{align*}
	 At last we take the terminal data as \begin{align*}	\begin{pmatrix}
	 \varphi_T\\ \psi_T
	 \end{pmatrix}= \hat{\Phi}_0=\begin{pmatrix}
	 0\\1
	 \end{pmatrix} .
	 \end{align*}
	 Then the identity \eqref{eq:identity1} gives \begin{align*}\int_{0}^{T}q(t) dt= -\ip{v_0}{1}_{(H_{per}^{1})^*,H_{per}^1}=:-\gamma_0.\end{align*} 
	 Thus, we have the following lemma:
	 \begin{lemma}\label{lem5.1}
	 	The control system \eqref{eq:cntrl2} is null controllable if and only if for any initial data $(u_0, v_0)\in \mathbf{H}^*$ with $\ip{u_0}{1}_{({H^2_{per})}^*,H^2_{per}}=0$, there exists a function $ q\in L^2(0,T)$ such that the following hold
	 	\begin{align}\label{bdrymom1}
	 		\begin{cases}
	 			\int_{-\frac{T}{2}}^{\frac{T}{2}}e^{-\overline{\lambda_k^\pm}t}q\left(t+\frac{T}{2}\right) =
	 			-e^{\overline{\lambda_k^\pm}\frac{T}{2}}\gamma_k^\pm ,\text{ for } k \in \z\setminus\{0\},\vspace{0.2cm} \\
	 			\int_{-\frac{T}{2}}^{\frac{T}{2}}q\left(t+\frac{T}{2}\right) dt=-\gamma_0,
	 		\end{cases}
	 	\end{align}
	 	where \begin{align*}
	 		&\gamma_k^\pm=\frac{1}{(ik^3+k^2-ik+\overline{\theta_k^\pm}\,)}\bigg[\ip{u_0}{e^{ikx}}_{(H_{per}^{2})^*,H_{per}^2}+ \overline{\theta_k^\pm} \ip{v_0}{e^{ikx}}_{(H_{per}^{1})^*,H_{per}^1}\bigg],\\ \text{ and }&\gamma_0 = \ip{v_0}{1}_{(H_{per}^{1})^*,H^1_{per}}.
	 	\end{align*}
	 \end{lemma} 
	\subsubsection{\textbf{Control acting on the heat component}}
	Let $(u_0, v_0)\in \mathbf{H^*}$ and $q\in L^2(0,T)$.
	Then using arguments similar to that of the last \Cref{sect5.1.1}, we get the following equivalent identity for null controllability: \begin{align}\label{eq:identity2}
	\ip{\begin{pmatrix}
		u_0\\ v_0
		\end{pmatrix}}{\begin{pmatrix}
		\varphi(0, \cdot)\\ \psi(0, \cdot)
		\end{pmatrix}}_{\mathbf{H^*},\mathbf{H}}
	+\int_{0}^{T}\left(\overline{\varphi(t, 2\pi)}+\overline{\psi(t,2\pi)}+\overline{\psi_x(t,2\pi)}\right)q(t) \,dt=0,
	\end{align}
	which leads to the following moment problem for the system \eqref{eq:cntrl2}:
	\begin{lemma}
		The control system \eqref{eq:cntrl2} is null controllable if and only if for any initial data $(u_0, v_0) \in \mathbf{H}^*$ with $\ip{u_0}{1}_{({H^2_{per})}^*,H^2_{per}}=0$ and $\ip{v_0}{1}_{{(H^1_{per})}^*,H^1_{per}}=0$, there exists a function $ q\in L^2(0,T)$ such that the following hold
		\begin{align}\label{bdrymom2}
				\int_{-\frac{T}{2}}^{\frac{T}{2}}e^{-\overline{\lambda_k^\pm}t}q\left(t+\frac{T}{2}\right)\, dt =
				-e^{\overline{\lambda_k^\pm}\frac{T}{2}}\gamma_k^\pm ,\text{ for } k \in \z\setminus\{0\},
		\end{align}
		where \begin{equation}\label{eq:gamma}\gamma_k^\pm=\frac{1}{\left(1+\overline{\theta_k^\pm}-ik \overline{\theta_k^\pm}\right)} \bigg[\ip{u_0}{e^{-ikx}}_{{(H^2_{per})}^*,H^2_{per}}+ \overline{\theta_k^\pm} \ip{v_0}{e^{-ikx}}_{{(H^1_{per})}^*,H^1_{per}}\bigg].
		\end{equation}
	\end{lemma} 

	\noindent
	From \Cref{fig6}, \Cref{fig7} given below, one can easily conclude that $(\,1+\overline{\theta_k^\pm}-ik\, \overline{\theta_k^\pm}\,)\neq 0,\, \forall k\in\z\setminus\{0\}$, and so $\gamma_k^\pm$ is well defined.
	\begin{figure}[H]
		\centering
		\begin{subfigure}[b]{0.49\textwidth}
			\includegraphics[width=1\textwidth]{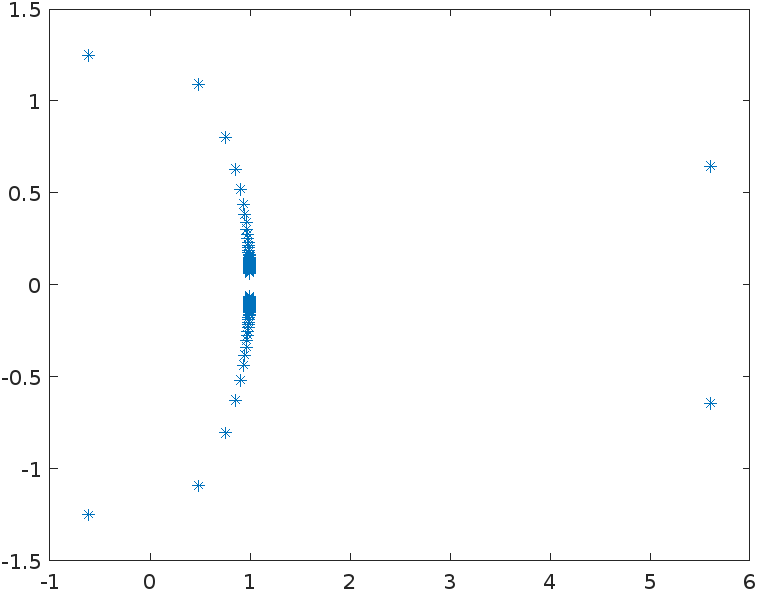}
			\caption{Assymptotic behaviour}
		\end{subfigure}
		\begin{subfigure}[b]{0.49\textwidth}
			\includegraphics[width=1\textwidth]{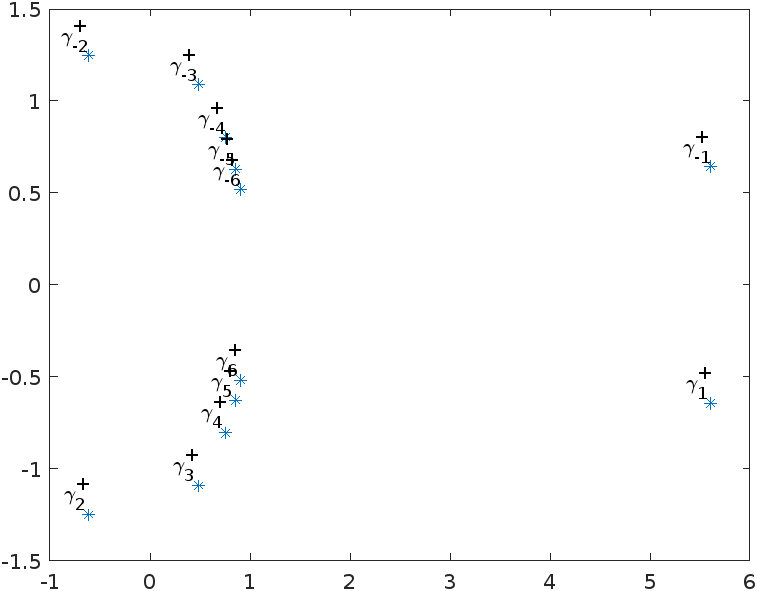}
			\caption{for $k\in[-6,6]\cap\z\setminus\{0\}$}
		\end{subfigure}
		\caption{Plot of $\lambda_k^++(1+ik)\eta_k^+$ to show $1+(1+ik)\theta_k^+\neq0$}
		\label{fig6}
	\end{figure}
	\begin{figure}[H]
		\centering
		\begin{subfigure}[b]{0.45\textwidth}
			\includegraphics[width=1\textwidth]{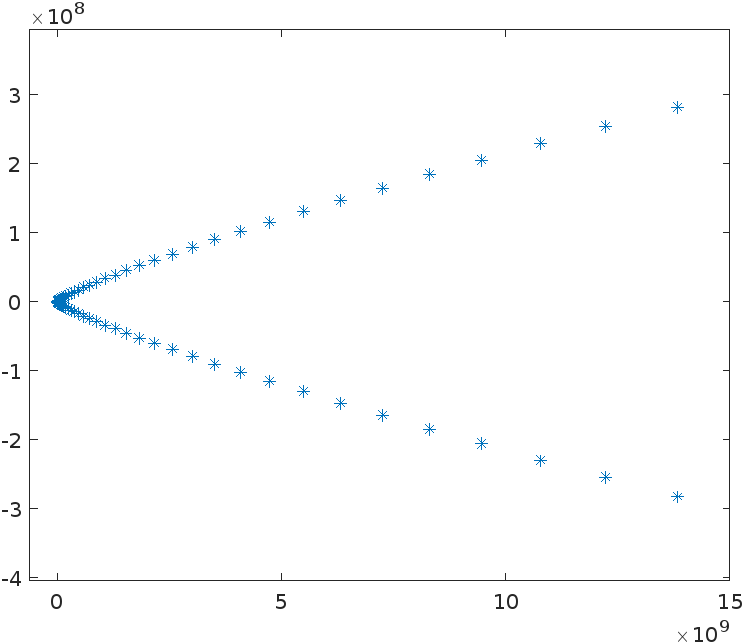}
			\caption{Assymptotic behaviour}
		\end{subfigure}
		\begin{subfigure}[b]{0.49\textwidth}
			\includegraphics[width=1\textwidth]{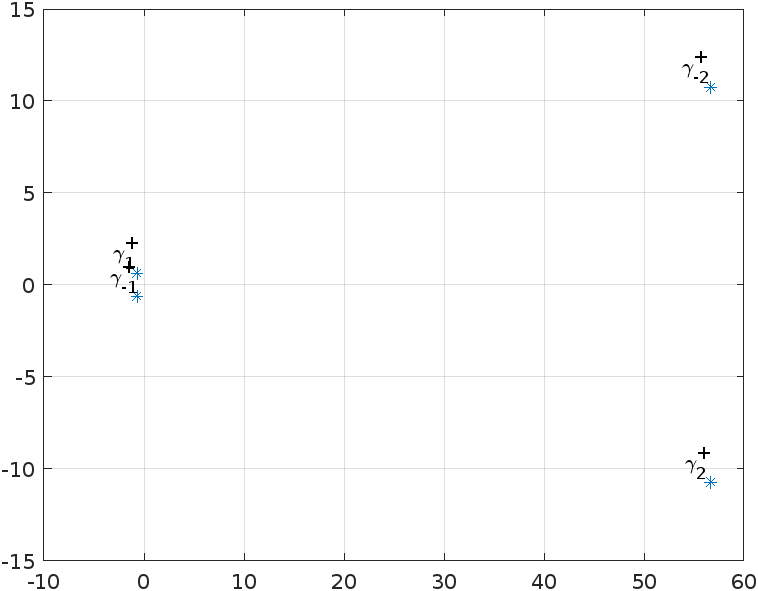}
			\caption{for $k\in[-2,2]\cap\z\setminus\{0\}$}
		\end{subfigure}
		\caption{Plot of $\lambda_k^-+(1+ik)\eta_k^-$ to show $1+(1+ik)\theta_k^-\neq0$}
		\label{fig7}
	\end{figure}

	\subsection{Proof of \Cref{thm2}}
	\begin{proof}
		As argued in \Cref{sect5.1.1}, we know that proving \Cref{thm2} is equivalent to solving the moment problem \eqref{bdrymom1}. So, let us define
		\begin{align*}
		q(t)=-\sum_{k\in \z\setminus\{0\}}e^{\overline{\lambda^+_k}\frac{T}{2}}\gamma_k^+\Theta_k^+\left(t-\frac{T}{2}\right)-\sum_{k\in \z\setminus\{0\}}e^{\overline{\lambda^-_k}\frac{T}{2}}\gamma_k^-\Theta_k^-\left(t-\frac{T}{2}\right)-\gamma_0\,\Theta_0\left(t-\frac{T}{2}\right).
		\end{align*}
		Using proposition \eqref{prop1}, one can easily conclude that $q$ solves the moment problem \eqref{bdrymom1}. The proof of the theorem is complete if we can show that $q\in L^2(0,T)$. 
		\begin{align*}
		&||q||_{L^2(0,T)}\leq C\left( \sum_{k\in \z\setminus\{0\}}\left|e^{\overline{\lambda_k^+}\frac{T}{2}}\right||\gamma_k^+|\,||\Theta_k^+||_{L^2(0,T)}+\sum_{k\in \z\setminus\{0\}}\left|e^{\overline{\lambda_k^-}\frac{T}{2}}\right||\gamma_k^-|\,||\Theta_k^-||_{L^2(0,T)}\right)\\
		&\hspace{15 mm}\leq C\left(||u_0||_{(H_{per}^{2})^*}+||v_0||_{(H_{per}^{1})^*}\right)\left( \sum_{k\in \z\setminus\{0\}}\frac{|k|^6}{|ik^3+k^2-ik+\overline{\theta_k^+}|}e^{-T |k|^2 -2\pi |k|^{\frac{1}{2}}+3(3+2\sqrt{2})\pi|k|}\right. \\ &\left.\hspace{47 mm}+\sum_{k\in \z\setminus\{0\}}\frac{|k|^7}{|ik^3+k^2-ik+\overline{\theta_k^-}|}\, e^{-T |k|^4 +(2\sqrt{2}+1)\pi |k|^2+\frac{T}{2}|k|^2+c\pi|k|}\right)\\
		&\hspace{15mm}<\infty.
		\end{align*}
	\end{proof}	
	\subsection{Proof of \Cref{thm2a}}
	\begin{proof}
		Similar to the previous proof, proving this theorem is equivalent to solving the moment problem \eqref{bdrymom1}. For that, we define
		\begin{align*}
			q(t)= -\sum_{k\in \z\setminus\{0\}}e^{\overline{\lambda^+_k}\frac{T}{2}}\gamma_k^+\Theta_k^+\left(t-\frac{T}{2}\right)-\sum_{k\in \z\setminus\{0\}}e^{\overline{\lambda^-_k}\frac{T}{2}}\gamma_k^-\Theta_k^-\left(t-\frac{T}{2}\right).
		\end{align*}
		Estimating this expression in $L^2(0,T)$ as done before, it is easy to see that $q\in L^2(0,T)$ and hence this $q$ solves the moment problem \eqref{bdrymom1}.
	\end{proof}
	\section{Construction of biorthogonal family of $\{1,e^{\lambda_k^\pm t}\}_{k \in \zahl\setminus\{0\}}$}\label{sect6}
	This section is devoted to the proof of the \Cref{prop1}. The entire study relies on the theory of logarithmic integral and complex analysis. One can refer to the books \cite{PK1}, \cite{PK} and \cite{RY} for detailed theory. Let us now define the exponential type functions
	\begin{definition}[Entire functions of exponential type]\quad\quad\quad\quad\quad

		\begin{itemize}
			\item An entire function $f$ is said to be of exponential type if there exists positive constants $A$ and $B$ such that
			\begin{align*}
			|f(z)|\leq A e^{B|z|}, \, z \in \cplx.
			\end{align*}
			\item An entire function $f$ is said to be of exponential type at most $\alpha\geq0$, if for any $\epsilon>0$, there exists a $A_{\epsilon}>0$ such that
			\begin{align*}
			|f(z)|\leq A_{\epsilon} e^{(\alpha+\epsilon) |z|}, \, z \in \cplx.
			\end{align*}
		\end{itemize}
	\end{definition}
	\begin{definition}(Sine type function)
		An entire function $f$ of exponential type $\pi$ is said to be of sine type if
		\begin{enumerate}[(i)]
			\item the zeros of $f(z)$, {$\mu_k$} satisfies gap condition, i.e., there exist $\delta>0$ such that $|\mu_k-\mu_l|>\delta$ for $k \neq l$, and
			\item there exist positive constants $C_1$,$C_2$ and $C_3$ such that
			$$C_1e^{\pi|y|}\leq |f(x+iy)| \leq C_2e^{\pi|y|},\,\forall \,x,y \in \rea \text{ with }|y|\geq C_3. $$
		\end{enumerate}
	\end{definition}
	The following proposition states some important properties of sine type functions:
	\begin{proposition}
		Let $f$ be a sine type function, and let $\{\mu_k\}_{k\in \mathcal{I}}$ with $\mathcal{I}\subset\z$ be its sequence of zeros. Then, we have:
		\begin{enumerate}[(a)]
			\item for any $\epsilon>0$, there exists constants $K_{\epsilon},\tl K_{\epsilon}>0$ such that
			$$K_{\epsilon}e^{\pi|y|}\leq |f(x+iy)|\leq \tl K_{\epsilon}e^{\pi|y|}, \text{ if dist$\left(x+iy,\{\mu_k\}\right)>\epsilon$},$$
			\item there exist some constants $K_1,K_2>0$ such that
			$$K_1<|f'(\mu_k)|<K_2, \quad \forall k\in\mathcal{I}.$$
		\end{enumerate}
	\end{proposition}
	Let us denote $\overline{\lambda_k^\pm}$ by $\mu_k^\pm$, for $k\in\z\setminus\{0\}$, and let $\mu_0=\mu_0^+=\mu_0^-=0$. 
	The main goal of this section is to find a class of entire functions $\mathcal{E}=\{\Psi_0,\Psi_k^\pm\}_{k\in \mathbb{Z}\setminus\{0\}}$ with the following properties:
	\begin{enumerate}
		\item[1.]	The family $\mathcal{E}$ contains entire functions of exponential type $\frac{T}{2}$, i.e., there exists a positive constant $C$ such that
		\begin{align}\label{exp typ}
			\abs{\Psi_0(z)}\leq C e^{\frac{T}{2}\abs{z}} \text{ and } 	\abs{\Psi_k^\pm(z)}\leq C e^{\frac{T}{2}\abs{z}}, \quad \forall z\in \cplx,  \quad \forall k\in\zahl\setminus\{0\}.
		\end{align}
		\item[2.] All the members of $\mathcal{E}$ are square integrable on the real line, i.e.
		\begin{align}\label{sq intg}
			\int_{\rea}\abs{\Psi_k^\pm(x)}^2<\infty \text{ and } \int_{\rea}\abs{\Psi_0(x)}^2< \infty.
		\end{align}
		\item[3.] The following relations hold 
		\begin{align}\label{bior}
			\begin{cases}
				\Psi_k^\pm(-i\mu_l^\pm)=\delta_{kl} \delta_{\pm},  \\
				\Psi_0(-i\mu_l^\pm)=\delta_{0l},
			\end{cases}\hspace{2 mm} k\in \zahl\setminus \{0\}, \, l \in \zahl
		\end{align}
	\end{enumerate}
Let us briefly explain how the existence of such family $\mathcal{E}$ of entire functions will give the desired biorthogonal, i.e., prove \Cref{prop1}. Thus for the time being, we assume the class of entire functions $\mathcal{E}$ with the properties \eqref{exp typ}, \eqref{sq intg} and \eqref{bior} exist.
	At first, let us state the celebrated Paley-Wiener theorem.
	\begin{theorem}[\textbf{Paley-Wiener}]
		Let $f$ be an entire function of exponential type $A$ and suppose
		\begin{align*}\int_{-\infty}^{\infty}|f(x)|^2 \, dx < \infty.
		\end{align*}
		Then there exists a function $\phi \in L^2(-A, A)$ such that:
		\begin{align*}
			f(z)=\int_{-A}^{A}e^{-izt}\phi(t) dt,\, z\in \cplx.
		\end{align*}
	\end{theorem}
Thus applying the Paley-Wiener theorem for the class $\mathcal{E}$ of entire functions, one can get a family of functions $\mathcal{B}=\{\Theta_0,\Theta_k^\pm\}_{k\in \zahl \setminus \{0\}}$ in $L^2(\rea)$ supported in $[-\frac{T}{2}, \frac{T}{2}]$, such that the following representation holds
	\begin{align}\label{representation}
		\Psi_k^\pm(z)=\int_{-\frac{T}{2}}^{\frac{T}{2}}e^{-izt}\Theta_k^\pm(t) dt,\, z\in \cplx.
	\end{align}
	Clearly, $\Theta_0,\Theta_k^\pm$ are the inverse Fourier transform of $ \Psi_0,\Psi_k^\pm$ respectively, $\forall\, k \in \zahl \setminus \{0\}.$
	Moreover, by Plancherel's Theorem, we have: 
	\begin{align*}
		\int_{\rea}\abs{\Psi_k^\pm(x)}^2 dx=2\pi\int_{-\frac{T}{2}}^{\frac{T}{2}}\abs{\Theta_k^\pm(t)}^2 dt.
	\end{align*}
	Note that, \eqref{bior} and the representation \eqref{representation} together imply the following
	\begin{align*}
		\begin{cases}
			\int_{-\frac{T}{2}}^{\frac{T}{2}}e^{-\mu_l^\pm t}\Theta_k^\pm \,dt &= \delta_{kl} \delta_{\pm}, \text{ for } k\in \zahl\setminus\{0\}, l\in \zahl 	\vspace{0.2cm},\\
			\int_{-\frac{T}{2}}^{\frac{T}{2}}e^{-\mu_l^\pm t}\Theta_0 \,dt &= \delta_{0l} \delta_{\pm}, \text{ for }  l\in \zahl,
		\end{cases}
	\end{align*}
	which essentially proves the \Cref{prop1} with $\mu_k^\pm=\overline{\lambda_k^\pm}.$

	Now we are at the position of constructing the family $\mathcal{E}$ satisfying \eqref{exp typ}, \eqref{sq intg} and \eqref{bior} to get the existence of desired biorthogonal family, $\mathcal{B}.$ Our analysis is inspired from the work \cite{LR14}.
	At first, we introduce the following entire function which has simple zeros exactly at $-i\mu_k^\pm$ ($k \in \z\setminus\{0\}$) and $-\mu_0$.  
	\begin{align}\label{eq:P}
	P(z)=z\prod_{k\in \z\setminus\{0\}}\left(1+\frac{z}{i\mu_k^+}\right) \prod_{k\in \z\setminus\{0\}}\left(1+\frac{z}{i\mu_k^-}\right).
	\end{align}
	\subsection{Estimating the canonical product P}
	\noindent
	In this section, we find some estimates related to $P$, as stated in the proposition below. These estimates are essential for the construction of biorthogonal family.

	\begin{proposition}
		Let $P$ be the canonical product defined in \eqref{eq:P}. Then $P$ is an entire function of exponential type at most zero satisfying the following estimates:
		\begin{align}
		\label{eq:1}	\abs{P(x)} &\leq C\abs{x}^{-1} e^{\sqrt{2}\pi \sqrt{\abs{x}}} e^{2\sqrt{2}\pi\abs{x}^{\frac{1}{4}}} , x\in \rea,\\
		\label{eq:2}	\abs{P'(-i\mu_k^+)}&\geq C |k|^{-3}e^{2\pi|k|^{\frac{1}{2}}}, k \in \zahl\setminus\{0\},\\
		\label{eq:3}		\abs{P'(-i\mu_k^-)}&\geq  C |k|^{-7}e^{3\pi |k|}, k \in \zahl\setminus\{0\}.
		\end{align}	
	\end{proposition}
	\begin{proof}For $k\in\z\setminus\{0\}$, we have
		\begin{align*}
		&\mu_k^+=-k^2-ik-iO(|k|^{-1})+O(|k|^{-2}),\quad \text{ as } \abs{k}\to \infty, \\& \mu_k^-= -k^4+ik^3+k^2-iO(|k|^{-1})+O(|k|^{-2}),\quad \text{ as } \abs{k}\to \infty,\\
		&\nonumber\mu_0=\mu^+_0=\mu^-_0=0.
		\end{align*}
		
		\noindent
		We consider $\arg z$ as the principal argument of a complex number $z \in \cplx \setminus \rea^-$, i.e., $\arg z \in (-\pi, \pi)$. So, for $z \in \cplx\setminus\rea^-$, we have
		\begin{align*}
		\log z= \log\abs{z}+ i \arg z,\quad \sqrt{z}=\sqrt{\abs{z}} \, e^{i\frac{\arg z}{2}}, \quad \sqrt[4]{z}=\sqrt[4]{\abs{z}} \, e^{i\frac{\arg z}{4}}. 
		\end{align*}
		For $k\in \z \setminus \{0\},$ we define
		\begin{align}\label{mutil}
		\widetilde \mu_k^+=\text{sgn}(k)\sqrt{-\mu_k^+}=k+ a_k,	\quad\widetilde \mu_k^-=\text{sgn}(k)\sqrt[4]{-\mu_k^-}=k+ b_k,
		\end{align}
		where $a_k=\frac{i}{2}+O(|k|^{-1})+iO(|k|^{-2}), \,\, b_k=-\frac{i}{4}+O(|k|^{-1})+i O(|k|^{-4}) \text{ as }|k|\rightarrow \infty$. 
		
		\noindent
		Let us now define
		\begin{align}
		\label{eq:p1}&P_1(z)=z\prod_{k\in \z\setminus\{0\}}\left(1+\frac{z}{i\mu_k^+}\right),\\
	\nonumber	&Q_1(z)=z \prod_{k\in \z\setminus\{0\}}\left(1-\frac{z}{\widetilde\mu_k^+}\right),\\
		\label{eq:p4}	&Q_2(z)=z^2\prod_{k\in \z\setminus\{0\}}\left(1+\frac{z^2}{\mu_k^+}\right).
		\end{align}
	$P_1$ is an entire function as the convergence of \eqref{eq:p1} is uniform in $z$ on any compact subset of $\cplx.$	 We also have the following relation:
		\begin{align}
		\label{eq:estp3}&Q_2(z)=-Q_1(z)Q_1(-z),\\
		\label{eq:estp1}&	P_1(z)=iQ_2(e^{-\frac{i \pi}{4}} \sqrt{z}).			
		\end{align}
		\begin{lemma}[Young \cite{RY}, Koosis \cite{PK} , Rosier\cite{LR14}]\label{Rosier}
			Let $\mu_k=k+d_k$, where $d_0=0,$ and $d_k=d+ O(k^{-1})$
			as $\abs{k} \to \infty$ for some constant $d \in \cplx$, and that $\mu_k\neq\mu_l$ for $k \neq l.$ Then $f(z)=z\prod_{k\in \z \setminus \{0\}} \left(1-\frac{z}{ \mu_k}\right)$ is an entire function of type sine.
		\end{lemma}
		\noindent
		Note that $\mu_k^+\neq\mu_l^+ , \text{ if } k\neq l \text{ as }\lambda_k^+\neq \lambda^+_l$(see \Cref{rem}) and also for large k, $\widetilde\mu_k^+$ is of the form $k+a_k$.
		Thanks to Lemma \Cref{Rosier}, we get $Q_1$ is an entire function of sine type (in particular, of exponential type $\pi$) and so for any $\epsilon>0$ we have:
		\begin{align}
		\label{bound}\abs{Q_1(z)}\leq C_1e^{\pi\abs{z}},&\:\: \forall z \in \mathbb{C}, \\
		\label{bound1}C_2e^{\pi\abs{y}}\leq \abs{Q_1(x+iy)} \leq C_3e^{\pi \abs{y}},\:\:& \text{if dist}(x+iy,\{\widetilde{\mu}_k^+\}) > \epsilon, \\
		\label{bound2}\abs{Q_1'(\widetilde{\mu}_k^+)}>C_4,&\:\: \forall k \in \z,
		\end{align}
		where $C_1,C_2,C_3,C_4$ are positive constants with $C_2,C_3$ depending on $\epsilon$. 

		\noindent
		Now using the estimates of $Q_1$ and the expression \eqref{eq:estp3}, we have:
		\begin{align}
		\label{q0}\abs{Q_2(z)}\leq C_1^2e^{2\pi\abs{z}},&\:\: \forall z \in \mathbb{C}, \\
		\label{q1} C_2^2e^{2\pi\abs{y}}\leq \abs{Q_2(x+iy)} \leq C_3^2e^{2\pi \abs{y}},\:\:& \text{if dist}\left(\pm(x+iy),\{\widetilde{\mu}_k^+\}_{k\in\z}\right)> \epsilon.
		\end{align}	Substituting $e^{-i\frac{\pi}{4}}\sqrt{z}$ in the place of $z=x+iy$ in \eqref{q0} and \eqref{q1}, respectively we obtain
		\begin{align}
		\label{p_1} |P_1(z)|&\leq C_1 e^{2\pi \sqrt{|z|}}, \forall z\in \cplx,\\
		\label{eq:p1imag} 
		C_2^2e^{2\pi\abs{\Im\left(e^{-i\frac{\pi}{4}}\sqrt{z}\right)}}\leq \abs{P_1(z)} &\leq  C_3^2e^{2\pi\abs{\Im\left(e^{\frac{-i\pi}{4}}\sqrt{z}\right)}},\:\:\text{if dist}\left(\pm \left(e^{-i\frac{\pi}{4}}\sqrt{z}\right),\{\widetilde{\mu}_k^+\}_{k\in\z}\right) > \epsilon.
		\end{align}
		Note that the estimate \eqref{eq:p1imag} holds true for large $x\in\rea$, and also one can get the estimate in a compact set using continuity of $P_1$. Thus combining both the facts, we get:
		\begin{equation}\label{eq:p1real}
		C_2^2e^{\sqrt{2}\pi \sqrt{\abs{x}}}\leq\abs{P_1(x)}\leq C_3^2e^{\sqrt{2}\pi \sqrt{\abs{x}}}, \quad\forall x\in\rea.
		\end{equation}
		To find the estimate of the last product term of the canonical product $P$, we define
		\begin{align}
		\label{eq:p2}&P_2(z)=z \prod_{k\in \z\setminus\{0\}}\left(1+\frac{z}{i\mu_k^-}\right),\\
		\label{eq:p6}&R_1(z)=z \prod_{k\in \z\setminus\{0\}}\left(1-\frac{z}{\widetilde \mu_k^-}\right),\\
		\label{eq:p5}&R_2(z)=z^4\prod_{k\in \z\setminus\{0\}}\left(1+\frac{z^4}{\mu_k^-}\right).
		\end{align}
		Hence we can write the following
		\begin{align}
		\label{eq:1p5}&R_2(z)=-R_1(z)R_1(-z)R_1(iz)R_1(-iz),\\
		\label{eq:1p2}	&P_2(z)=iR_2(e^{-i\frac{\pi}{8}} \sqrt[4]{z}).
		\end{align}		
		Note that $\widetilde \mu_k^-\neq\widetilde \mu_l^-\text{ if }k\neq l \text{ as } \lambda_k^-\neq\lambda_l^- $(see \Cref{rem}) and also for large $k$, $\widetilde \mu_k^-$ is of the form $k+b_k $. Thus, $R_1$ is an entire function of sine type, thanks to \Cref{Rosier} again, and so for any $\epsilon>0$ there exists positive constants $C_5,C_6,C_7,C_8$, where $C_6,C_7$ depends on $\epsilon$ such that:
		\begin{align}
		\nonumber&|R_1(z)|\leq C_5 e^{\pi |z|},\forall z\in \cplx,\\
		\label{r2}	C_6e^{\pi \abs{y}} \leq &\abs{R_1(x+iy)} \leq C_7 	e^{\pi \abs{y}}, \text{ if } \dist\left(x+iy, \{\widetilde \mu_k^-\}\right)> \epsilon,\\
		\label{r22}	& \, \abs{R_1'(\widetilde \mu_k^-)} \geq C_8, \,\ 	k\in \z.
		\end{align} 
		Using these estimates of $R_1$ and the expression \eqref{eq:1p5} we get
		\begin{align}
		\nonumber	&\quad \quad \quad \quad \quad \quad \quad\quad 	\quad \quad\abs{R_2(z)}\leq C_5^4 e^{4\pi \abs{z}},\forall z \in \cplx\\
		\label{eq:p5lower}&C_6^4 e^{2\pi\left(\abs{x}+\abs{y}\right)}\leq \abs{R_2(x+iy)}\leq C_7^4 e^{2\pi\left(\abs{x}+\abs{y}\right)},\text{ if } \dist(\{\pm(x+iy), \pm(-y+ix)\}, \{\widetilde \mu_k^-\})> \epsilon. 
		\end{align} 
		At last using the relation \eqref{eq:1p2} we get the following bound for $P_2$ 
		\begin{align}
		&\label{eq:p2est3}\quad \quad \quad \quad \quad \quad \quad\quad \quad \quad|P_2(z)|\leq C_5^4e^{4\pi\sqrt[4]{|z|}}, \forall z \in \cplx,\\
		&\label{P2lowup}C_6^4e^{2\pi\left(|\Re(e^{-i\frac{\pi}{8}} \sqrt[4]{z})|+|\Im(e^{-i\frac{\pi}{8}} \sqrt[4]{z})|\right)}\leq 	\abs{P_2(z)}\leq C_7^4e^{2\pi\left(|\Re(e^{-i\frac{\pi}{8}} \sqrt[4]{z})|+|\Im(e^{-i\frac{\pi}{8}} \sqrt[4]{z})|\right)},
		\end{align}
		provided $\dist\left(\{\pm e^{-i\frac{\pi}{8}}\sqrt[4]{z},\pm i e^{-i\frac{\pi}{8}}\sqrt[4]{z}\} \,, \{\widetilde \mu_k^-\} \right)> \epsilon.$
		Similar to the case of $P_1$, using inequality \eqref{P2lowup} and the continuity of $P_2$, we get the following bound of $P_2$ on $\rea$:
		\begin{equation}\label{eq:p2real}
		\abs{P_2(x)}\leq C e^{2\pi\left(|\cos(\frac{\pi}{8})|+\sin(\frac{\pi}{8})|\right)\abs{x}^{\frac{1}{4}}}\leq C e^{2\sqrt{2}\pi\abs{x}^{\frac{1}{4}}}.
		\end{equation}
		
		\noindent
		To get the estimate on $P$, first note that $P$ satisfies the relation: \begin{align}\label{P}
		P(z)=\frac{P_1(z)P_2(z)}{z}, \forall z\in\cplx.
		\end{align}
		Thus, using \eqref{eq:p1}, \eqref{eq:p2est3} and the above relation, we get:
		\begin{align*}
		\abs{P(z)}\leq \frac{C}{|z|} e^{2\pi \sqrt{\abs{z}}} e^{4\pi {\sqrt[4]{|z|}}}\leq C e^{\tl \epsilon |z|}\frac{1}{|z|}e^{2\pi \sqrt{\abs{z}}+4\pi {\sqrt[4]{|z|}}-\tl \epsilon|z|}&\leq C(\tl \epsilon) e^{\tl \epsilon |z|},\,z\in\cplx,
		\end{align*} where $\tl \epsilon>0$ is any number.
		So, $P$ is an entire function of exponential type at most 0.
		
		\noindent
		Also, combining \eqref{eq:p1real}, \eqref{eq:p2real} and \eqref{P} we obtain \eqref{eq:1}.
		
		\noindent
		We now establish the estimates \eqref{eq:2}, \eqref{eq:3}.
		Performing differentiation on \eqref{P} we get 
		\begin{align}\label{eq:pprime}
		P'(z)=\frac{P_1'(z)P_2(z)}{z}+\frac{P_1(z)P_2'(z)}{z}-\frac{P_1(z)P_2(z)}{z^2}.
		\end{align}
	Using the fact $P_1(-i\mu_k^+)=0,$ we infer from  \eqref{eq:pprime} \begin{align}\label{pprime}
		P'(-i\mu_k^+)=\frac{P_1'(-i\mu_k^+)P_2(-i\mu_k^+)}{-i\mu_k^+}.
	\end{align}
		From estimates  \eqref{eq:estp3}, \eqref{eq:estp1} we further have
		\begin{align}\label{pprime1}
		P_1'(z)=\frac{i e^{-\frac{i \pi}{4}}}{2\sqrt{z}}\bigg[-Q_1'(e^{-\frac{i \pi}{4}}\sqrt{z})Q_1(-e^{-\frac{i \pi}{4}}\sqrt{z})+Q_1(e^{-\frac{i \pi}{4}}\sqrt{z})Q_1'(-e^{-\frac{i \pi}{4}}\sqrt{z})\bigg].
		\end{align}
		
		\noindent
		Note that
		$\text{sgn}(k)e^{-\frac{i \pi}{4}}\sqrt{-i\mu_k^+}= \, \widetilde \mu_k^+$ are zeros of $Q_1$ and so we have
		\begin{align*}
		P_1'(-i\mu_k^+)&=-\frac{1}{2\tl\mu_k^+}\bigg[Q_1'\left(\widetilde \mu_k^+ \right)Q_1\left(-\widetilde \mu_k^+\right)\bigg].
		\end{align*}
	From the asymptotic expression of $\widetilde\mu_k^+$ \eqref{mutil}, it is clear that there exist a $\delta_1>0$ such that $|\widetilde \mu_k^+ +\widetilde \mu_l^+|>\delta_1$ for large $k$ and $l$. For the remaining finite number of $k,l\in\z$, if $k\neq l$ we know that $\widetilde\mu_k^+\neq -\widetilde\mu_l^+$ as $\lambda_k^+\neq \lambda_l^+$(see \cref{fig2}), and if $k=l\neq 0$ then obviously $|\widetilde \mu_k^+ +\widetilde \mu_l^+|>\delta_2$ for some $\delta_2>0$. Thus, combining all these facts we get  existence of a $\delta>0$ such that $|\widetilde \mu_k^+ +\widetilde \mu_l^+|>\delta$ for $k\in\z\setminus\{0\},\, l\in\z$ and so using \eqref{bound1} in the above relation, we get
	\begin{align}\label{P_1'}
		|P_1'(-i\mu_k^+)|\geq C\frac{1}{|k|},\,\forall k\in\z\setminus\{0\}.\end{align}
	Next we estimate $\abs{P_2(-i\mu_k^+)}$ using \eqref{P2lowup} to get the lower bound of $|P'(-i\mu_k^+)|$.
	Let us first compute
	\begin{align*}
	\bigg|\Re\left(e^{-i\frac{\pi}{8}}\sqrt[4]{-i \mu_k^+}\right)\bigg|&=|k|^{1/2}+O(|k|^{-2}), \text{ as }|k|\rightarrow\infty, \text{ and }\\
	\bigg|\Im\left(e^{-i\frac{\pi}{8}}\sqrt[4]{-i \mu_k^+}\right)\bigg|&=O(|k|^{-1/2}), \text{ as }|k|\rightarrow\infty.
	\end{align*}
	Note that $\pm\sqrt[4]{-\mu_k^+}\neq \widetilde\mu_l^-$ and $\pm i\sqrt[4]{-\mu_k^+}\neq \widetilde\mu_l^-$ as $\lambda_k^+\neq \lambda_l^-$ for $k\in\z\setminus\{0\},l\in\z$, and also asymptotically the gap between $\{\pm\sqrt[4]{-\mu_k^+},\pm i\sqrt[4]{-\mu_k^+}\}$ and $\widetilde\mu_l^-$ increases as $k,l$ increases.
	Thus, there exist a $\delta>0$ such that dist$\left(\{\pm\widetilde\mu_k^+,\pm i\widetilde\mu_k^+\}_{k\in\z\setminus\{0\}},\{\widetilde\mu_l^-\}_{l\in\z}\right)>0$ and so by \eqref{P2lowup} we get a $C>0$ such that
		\begin{align}\label{P_2}
		\abs{P_2(-i\mu_k^+)}&\geq Ce^{2\pi\left(|\Re(e^{-i\frac{\pi}{8}} \sqrt[4]{-i \mu_k^+})|+|\Im(e^{-i\frac{\pi}{8}} \sqrt[4]{-i \mu_k^+})|\right)} \geq C e^{2\pi|k|^{\frac{1}{2}}}.
		\end{align}
		Using the bounds \eqref{P_1'},\eqref{P_2} in \eqref{pprime}, we get: \begin{align*}
		|P'(-i\mu_k^+)|\geq C\frac{1}{|\mu_k^+|}|k|^{-1}e^{2\pi|k|^{\frac{1}{2}}}\geq C |k|^{-3}e^{2\pi|k|^{\frac{1}{2}}}.
		\end{align*}
		
		Now, from \eqref{eq:pprime} we have   \begin{align}\label{p pr ex}
		P'(-i\mu_k^-)=\frac{P_1(-i\mu_k^-)P_2'(-i\mu_k^-)}{-i\mu_k^-}.
		\end{align}
		Recall that for $k\in\z\setminus\{0\},\text{ sgn}(k)e^{-\frac{i \pi}{8}}\sqrt[4]{-i\mu_k^-}=\widetilde \mu_k^-$ are zeros of $R_1$, and so using the relations \eqref{eq:1p5} and \eqref{eq:1p2}, we get: 
		\begin{align*}
		|P_2'(-i\mu_k^-)|&=\frac{|i e^{-\frac{i \pi}{8}}|}{|4\sqrt[3/4]{-i\mu_k^-}|}|\,\text{sgn}(k)R_1'(\widetilde \mu_k^-)R_1(i \widetilde \mu_k^-)R_1(-\widetilde \mu_k^-)R_1(-i \widetilde \mu_k^-)
		|.
		\end{align*}
 Arguing in the same way as done earlier to get the estimate \eqref{P_1'}, we get existence of a $\delta>0$ such that ${|\widetilde \mu_m^-+ \widetilde \mu_k^-|>\delta}, |\widetilde \mu_m^-+i \widetilde \mu_k^-|>\delta, | \widetilde \mu_m^--i\widetilde \mu_l^-|>\delta,$ for all $k \in \zahl \setminus \{0\}, \, m \in \zahl.$ 
 	Thus, using the bounds \eqref{r2} and \eqref{r22} in the above relation, we get
 	\begin{align}\label{p2 prime exp}
 		|P_2'(-i\mu_k^-)| 		&\geq C \frac{1}{|4\sqrt[3/4]{-i\mu_k^-}|}e^{2\pi |k|}.
 	\end{align}
		It remains to estimate the term $|P_1(-i\mu_k^-)|$. Using the argument used to obtain \eqref{P_2}, we get existence of a $\delta>0$ such that
		\begin{align}\label{2} \bigg|\left(e^{\frac{-i\pi}{4}}\sqrt{-i\mu_k^-}\right)-\widetilde{\mu}_l^+ \bigg|> \delta , \text{ for all } k \in \zahl \setminus \{0\}, \, l \in \zahl. \end{align}
	By \eqref{eq:p1imag}, we have 
		\begin{align}\label{p1 exp}
		|P_1(-i\mu_k^-)|\geq C e^{2\pi\abs{\text{Im}\left(e^{\frac{-i\pi}{4}}\sqrt{-i\mu_k^-}\right)}} =Ce^{2\pi\abs{\text{Im}\left(\sqrt{-\mu_k^-}\right)}}.
		\end{align} 
Hence, using the bounds \eqref{p2 prime exp} and \eqref{p1 exp} in the absolute expression of \eqref{p pr ex}  and noting the fact that $|\Im(\sqrt{-\mu_k^-})|=\frac{|k|}{2}+O(|k|^{-1})$, one can write the final estimate as
		\begin{align}\label{3}
		\nonumber|P'(-i\mu_k^-)|&\geq C \frac{1}{|4\sqrt[7/4]{-i\mu_k^-}|}e^{2\pi |k|}e^{\pi |k|}\\
		&\geq C |k|^{-7}e^{3\pi |k|},
		\end{align}
		which proves \eqref{eq:3}.
	\end{proof}
	\subsection{Construction of the multiplier}
	In this section, we construct two Beurling and Malliavin's multipliers $M_1$ and $M_2$ (see \cite{BM} and \cite{PK1} for more details) which compensates the exponential growth of $P_1$ and $P_2$, respectively on $\rea$. More precisely, we construct a multiplier $M=M_1 M_2$ such that $P(x)M(x)$ is bounded on the real line. Here we mainly follow the works \cite{AM21} for finding $M_2$,  and \cite{OG} and \cite{LR14} for $M_1$.   
	
Define $a_1=a_2=a=\frac{T}{4\pi}>0$, $b_1=\sqrt{2}>0$ and $b_2=\frac{2\sqrt{2}\sqrt{\pi}}{\cot(\frac{\pi}{8})} \frac{\Gamma\left(\frac{9}{8}\right)}{\Gamma\left(\frac{5}{8}\right)}>0$. Let us first construct the multiplier $M_2$.
	\begin{proposition}[Multiplier $M_2$]\label{m prop}
		There exists an entire functions $M_2$ of exponential type at most $a\pi$ which satisfies the following
		\begin{align}
		&\label{m_1bound1}	|M_2(x)|\leq C_1\, |x|\, e^{-2\sqrt{2}\pi|x|^\frac{1}{4}}, \\
		&\label{m_1bound2} |M_2(-i\mu_k^+)|\geq C_2\, e^{\pi a |k|^2-4(\sqrt{2}+1)\pi|k|}, k\in\z\setminus\{0\},\\
		&\label{m_1bound3}|M_2(-i\mu_k^-)| \geq C_3\, e^{\pi a k^4 -c|k|}, k\in\z\setminus\{0\},
		\end{align}
		for some $c,\,C_1,C_2,C_3>0$.
	\end{proposition}
	\begin{proof}We split the proof into two parts. At first we give construction of the multiplier, $M_2$ following \cite{AM21}. In the second part, we obtain some additional estimates on $M_2$ to get the desired estimates at the points $i\mu_k^+$ and $i\mu_k^-$, which is crucial for the construction of required biorthogonal family.
		
		\noindent
\textbf{Construction of $M_2$.}
		Let us consider the function defined on the ray $(0, \infty)$ as
		\begin{align*}
		s_2(t)=at-b_2 t^{\frac{1}{4}}, \hspace{5 mm} t>0.
		\end{align*}
		For $\gamma \in (0,2)$, we have (see \cite{AM21}):
		\begin{align*}
		\int_{0}^{\infty} \log\left|1-\frac{x^2}{t^2}\right|\, dt^{\gamma} = |x|^{\gamma} \pi \cot \frac{\pi \gamma}{2}, \, \forall x \in \rea.
		\end{align*}
	Thus, the above expression and the definition of $s_2$ imply the following
		\begin{align}\label{eq:log}
		\int_{0}^{\infty} \log\left|1-\frac{x^2}{t^2}\right|\, ds_2(t) = -b_2 |x|^{\frac{1}{4}}\pi \cot\left(\frac{\pi}{8}\right), \forall x \in \rea.
		\end{align}
		Note that:
		\begin{itemize}
			\item $s_2(B_2)=0$ for $B_2=\left(\frac{b_2}{a}\right)^{\frac{4}{3}}$,
			\item $s_2(t)$ is increasing for $t>\left(\frac{1}{4}\right)^\frac{4}{3}B_2$.
		\end{itemize}
		Next we define 
		\begin{align*}
		\nu_2(t)=\begin{cases}
		0, & t\leq B_2,\\
		s_2(t), & t\geq B_2.
		\end{cases}
		\end{align*}
		Then, $\nu_2$ is non-negative and non-decreasing function. Moreover, $0\leq \nu_2(t) \leq at$ for $t>0$. 
		Let us define 
		\begin{align}
	\nonumber	g_2(z)&=\int_{0}^{\infty} \log\left(1-\frac{z^2}{t^2}\right)\, d\nu_2(t) = \int_{B_2}^{\infty} \log\left(1-\frac{z^2}{t^2}\right)\, ds_2(t), \hspace{2mm}  z\in \cplx \setminus \rea, \\
	\label{u2}	U_2(z)&= \int_{0}^{\infty} \log\left|1-\frac{z^2}{t^2}\right|\, d\nu_2(t) = \int_{B_2}^{\infty} \log\left|1-\frac{z^2}{t^2}\right|\, ds_2(t), \hspace{2mm}  z\in \cplx .
		\end{align}
		Note that 
		\begin{itemize}
			\item $g_2$ is holomorphic function on $\cplx\setminus\rea$ and $U_2$ is continuous on $\cplx.$
			\item $U_2(z)=\Re(g_2(z))$ for $z\in \cplx\setminus\rea$ as $\log(z)=\log|z|+i \arg(z)$.
			\item If we atomize the measure $d\nu_2$(i.e., to make the measure integer valued), then the integral $g_2$ would become logarithm of an entire function of exponential type and hence taking exponential of $g_2$ would give us entire function of exponential type.
		\end{itemize}
		Thus, let us now define 
		\begin{align*}
		\tilde{g}_2(z)&=\int_{0}^{\infty} \log\left(1-\frac{z^2}{t^2}\right)\, d[\nu_2(t)] = \int_{B_2}^{\infty} \log\left(1-\frac{z^2}{t^2}\right)\, d[s_2(t)], \hspace{2mm} & z\in \cplx \setminus \rea \\
		\tilde{U}_2(z)&= \int_{0}^{\infty} \log\left|1-\frac{z^2}{t^2}\right|\, d[\nu_2(t)] = \int_{B_2}^{\infty} \log\left|1-\frac{z^2}{t^2}\right|\, d[s_2(t)], \hspace{2mm} & z\in \cplx,
		\end{align*}
	where the notation $[x]$ is for the integral part of $x$. As above, $\tl{g}$ is holomorphic function on $\cplx\setminus\rea$ and $\tl U$ is continuous on $\cplx.$
		Thus we can write
		\begin{align*}
		\exp(\tilde{g}_2(z))=\prod_{k\in \nat}\left(1-\frac{z^2}{\tau_k^2}\right), \text{ for }z \in \cplx\setminus\rea,
		\end{align*}
		where $\tau_k$'s are th point of discontinuity of the map $t \to[v(t)]$.
		
	\noindent	
	Let us define multiplier $M_2$ as $M_2(z)=\exp(\tilde{g}_2(z-i))$. So, for $z\in\cplx\setminus\rea$ we have $|M_2(z)|=\exp(\tilde{U}_2(z-i))$ as $\tilde{U}_2(z-i)=\Re(\tilde{g}_2(z-i))$.
	
		\noindent	
		\textbf{Estimates on the multiplier $M_2$.}
		\noindent
		Here we find upper and lower bound for $\tilde{U}_2$. This estimation is done in two steps:
		
		\noindent
		\emph{Step 1.} Estimate on $U_2(z)$.
		\begin{lemma}[see \cite{AM21}]\label{lemma glass}
			Let $U_2$ be defined by \eqref{u2}. Then there exists $\theta\in L^{\infty}(\rea^+)$ such that the following identity holds
			\begin{align}
				\label{u rea}	&U_2(x)+b_2\,\pi\, \cot\left(\frac{\pi}{8}\right)\sqrt[4]{|x|}=-aB_2\theta\left(\frac{|x|}{B_2}\right), \,  x \in \rea.
			\end{align}
		\end{lemma}
		For the sake of completeness we give a sketch of the proof of this lemma in Appendix A (\Cref{app}).
		\begin{lemma}[see \cite{AM21}, Lemma 17] 
			The function $U_2$ defined in \eqref{u2} satisfies the following relation
			\begin{align}\label{u cplx}	
				U_2(z)=|\Im(z)|\left(\pi a + \frac{1}{\pi}\int_{-\infty}^{\infty}\frac{U_2(t)}{|z-t|^2}\right), \, z\in \cplx\setminus\rea.
			\end{align}
		\end{lemma}
Without loss of generality, we assume $\Im(z)>0$. Then plugging \eqref{u rea} in the second term of the right hand side of \eqref{u cplx} we obtain
	\begin{align}\label{int_term}
		\frac{y}{\pi}\int_{-\infty}^{\infty}\frac{U_2(t)}{|z-t|^2}\, dt=-\frac{y}{\pi}\int_{-\infty}^{\infty}\frac{b_2\,\pi\, \cot\left(\frac{\pi}{8}\right)}{(x-t)^2+y^2}|t|^{\frac{1}{4}}\, dt-\frac{y}{\pi}a B_2 \int_{-\infty}^{\infty}\frac{\theta\left(\frac{|t|}{B_2}\right)}{(x-t)^2+y^2}\, dt.
	\end{align}
	Then the second term of \eqref{int_term} can be estimated as
	\begin{align}\label{term1}
		\left|\frac{y}{\pi}aB_2 \int_{-\infty}^{\infty}\frac{\theta\left(\frac{|t|}{B_2}\right)}{(x-t)^2+y^2}\, dt\right|\leq ||\theta||_{\infty} aB_2\frac{y}{\pi} \int_{-\infty}^{\infty}\frac{ds}{y\left(\left(\frac{x}{y}-s\right)^2+1\right)}=aB_2||\theta||_{\infty}.
	\end{align}
		Next, we write the first term of \eqref{int_term} as following
		\begin{align*}
		\frac{y}{\pi}\int_{-\infty}^{\infty}\frac{(-b_2\,\pi) \cot\left(\frac{\pi}{8}\right)}{(x-t)^2+y^2}|t|^{\frac{1}{4}}\, dt &= -b_2\, y\, \cot\left(\frac{\pi}{8}\right) \int_{-\infty}^{\infty}\frac{|y|^{\frac{1}{4}} |s|^{\frac{1}{4}}}{(x-ys)^2+y^2}\,y\,ds\nonumber \\ 
		&=-b_2\,\cot\left(\frac{\pi}{8}\right) y^{\frac{1}{4}}\int_{-\infty}^{\infty}\frac{|s|^{\frac{1}{4}}}{\left(\frac{x}{y}-s\right)^2+1} \,ds.
		\end{align*}
	An elementary analysis will give us the following inequality (see Appendix B (\Cref{secinq}) for details)
	\begin{align}\label{ineq}
		2\sqrt{\pi} (1+x^2)^{\frac{1}{8}}  \,    \frac{\Gamma\left(\frac{5}{8}\right)}{\Gamma\left(\frac{9}{8}\right)}\leq\int_{-\infty}^{\infty}\frac{|s|^{\frac{1}{4}}}{1+|x-s|^2}\,ds\leq 2\sqrt{\pi} (1+x^2)^{\frac{1}{8}} \,    \frac{\Gamma\left(\frac{3}{8}\right)}{\Gamma\left(\frac{7}{8}\right)}.
	\end{align}
Using \eqref{term1} and \eqref{ineq} in \eqref{int_term}, we obtain
\begin{align*} 
	-b_2\,\cot\left(\frac{\pi}{8}\right) \frac{\Gamma\left(\frac{3}{8}\right)}{\Gamma\left(\frac{7}{8}\right)} 2\sqrt{\pi} (y^2+x^2)^{\frac{1}{8}}-&aB_2||\theta||_{\infty} \leq \frac{y}{\pi}\int_{\infty}^{\infty}\frac{U_2(t)}{|z-t|^2}\, dt\ \nonumber\\ & \leq -b_2\, \cot\left(\frac{\pi}{8}\right) \frac{\Gamma\left(\frac{5}{8}\right)}{\Gamma\left(\frac{9}{8}\right)} 2 \sqrt{\pi} (y^2+x^2)^{\frac{1}{8}} +aB_2||\theta||_{\infty}.
\end{align*}  
Finally using the definition of $U_2(z)$ and the fact that $b_2 \,\cot\left(\frac{\pi}{8}\right) \frac{\Gamma\left(\frac{5}{8}\right)}{\Gamma\left(\frac{9}{8}\right)}=2\sqrt{2}\sqrt{\pi}$, we get
\begin{align}\label{U_2 bound}
	-4\sqrt{2}\pi \frac{\Gamma\left(\frac{3}{8}\right)}{\Gamma\left(\frac{7}{8}\right)}\frac{\Gamma\left(\frac{9}{8}\right)}{\Gamma\left(\frac{5}{8}\right)}(y^2+x^2)^{\frac{1}{8}}-aB_2||\theta||_{\infty} &\leq U_2(z)-\pi a|y|\nonumber\\ & \leq -4\sqrt{2}\pi (y^2+x^2)^{\frac{1}{8}} +aB_2||\theta||_{\infty}.
\end{align}
\emph{Step2.} Bound for $\tilde{U}_2$.

\noindent
We first find the bound on $\tl U_2- U_2$. This estimate when combined with that of $U_2$ will give the required bound for $\tl U_2$ and hence the desired estimate of the multiplier $M_2$. 
			\begin{lemma}[see \cite{LR14}, Lemma 3.8]\label{lem}
			Let $\nu:\rea^+\rightarrow \rea^+$ be nondecreasing and null on $(0,B)$. Then, for $z=x+iy \in \cplx\setminus \rea$, we have:
			\begin{align*}
				-\log^+\frac{|x|}{|y|}-\log^+\frac{x^2+y^2}{B^2}-\log2 \leq \tilde{U}(x+iy)-U(x+iy)\leq \log^+\frac{|x|}{|y|}.
			\end{align*} 
		\end{lemma}
		Clearly, $\nu_2$ satisfies the hypothesis of above lemma, and so we have:
		\begin{align}\label{u2-}
			-\log^+\frac{|x|}{|y|}-\log^+\frac{x^2+y^2}{B_2^2}-\log2 \leq \tilde{U}_2(x+iy)-U_2(x+iy)\leq \log^+\frac{|x|}{|y|}.
		\end{align} 
		On combining the estimates \eqref{U_2 bound} and \eqref{u2-} obtained above and denoting the term $\frac{\Gamma\left(\frac{3}{8}\right)}{\Gamma\left(\frac{7}{8}\right)}\frac{\Gamma\left(\frac{9}{8}\right)}{\Gamma\left(\frac{5}{8}\right)}=\hat{c}>0$, we can find the bounds for $\tilde{U}_2$, which is:
		\begin{align}
			-4\sqrt{2}\pi \hat{c} 
			(y^2+x^2)^{\frac{1}{8}}-aB_2||\theta||_{\infty}+\pi &a|y|-\log^+\frac{|x|}{|y|}-\log^+\frac{x^2+y^2}{B^2}-\log2\nonumber\\ \label{tl U_2}\leq \tilde{U}_2(z)
			\leq -4\sqrt{2}\pi (y^2+x^2)^{\frac{1}{8}}&+aB_2||\theta||_{\infty}+\log^+\frac{|x|}{|y|}+\pi a|y|.
		\end{align}
		Using \eqref{tl U_2}, we can establish the desired estimates of $M_2$. Let us first find the upper bound $M_2$ on $\rea$. 
		\begin{align*}
			|M_2(x)|=\exp(\tilde{U}_2(x-i))&\leq \exp(-4\sqrt{2}\pi(1+x^2)^{\frac{1}{8}}+aB_2||\theta||_{\infty}+\log^+(|x|))\nonumber\\
			& \leq C\, \exp\left(-4\sqrt{2}\pi |x|^\frac{1}{4}+\log^+|x|+aB_2||\theta||_{\infty}\right)\nonumber\\
			& \leq C\, |x|\, \exp(-4\sqrt{2}\pi|x|^\frac{1}{4}),
		\end{align*}
		which proves \eqref{m_1bound1}.
		
		\noindent
		We now estimate $M_2$ at the zeros of $P$, i.e., we prove \eqref{m_1bound2} and \eqref{m_1bound3}. Recall 
		\begin{align*}
			\mu_k^{+}&=-k^2+ik-iO(|k|^{-1})+O(|k|^{-2}), \\
			\mu_k^{-}&=-k^4+ik^3+k^2-iO(|k|^{-1})-O(|k|^{-2}).
		\end{align*}
		So, we have
		\begin{align*}
			\tilde{U}_2(-i\mu_k^{+}-i)&=\tilde{U}_2((-k-O(|k|^{-1}))+i(k^2-1-O(|k|^{-2})))\\
			\tilde{U}_2(-i\mu_k^--i)&=\tilde{U}_2((k^3-O(|k|^{-1}))+i(k^4-k^2-1+O(|k|^{-2}))) .
		\end{align*}
		Thus, inequality \eqref{tl U_2} gives 
		\begin{align*}
			|M_2(-i\mu_k^+)|&=\exp(\tilde{U}_2(-i\mu_k^+-i))\nonumber\\
			&\geq \exp\left(-4\sqrt{2}\hat{c}\pi(k^4-k^2+O(1))^{\frac{1}{8}}-aB_2||\theta||_{\infty}+a\pi|k^2-1-O(|k|^{-2})|\right)\nonumber\\ & \hspace{25 mm} \exp\left(-\log^+\frac{|k+O(|k|^{-1})|}{|k^2-1-O(|k|^{-2})|}-\log^+\frac{k^4-k^2+O(1)}{B_2^2}\right)\nonumber\\
			& \geq C\, \exp\left(-4\sqrt{2}\hat{c}\pi(k^4)^{\frac{1}{8}}+\pi a |k|^2 -\log\left(c\,k^4\right)\right), \text{ for some }c,C >0 \nonumber\\
			& \geq C \, \exp(-4\sqrt{2} \hat{c}\pi |k|^{\frac{1}{2}}+\pi a |k|^2 -4 \log|c\,k|)\nonumber\\
			& \geq \exp(-4\sqrt{2}\hat{c} \pi|k|^{\frac{1}{2}}+\pi a |k|^2-4|k|)\nonumber\\
			& \geq \exp(\pi a |k|^2-4(\sqrt{2}+1)\pi|k|), 
		\end{align*}
		which is \eqref{m_1bound2}, and 
		\begin{align*}
			|M_2(-i\mu_k^-)|&=\exp\left(\tilde{U}_2(-i\mu_k^--i)\right)\nonumber \\
			&\geq C\, \exp\left(-4 \sqrt{2}\hat{c}\pi(c_1\,k^8)^{\frac{1}{8}}+\frac{\pi a}{2}k^4-\log(c_2\,k^8)\right),\text{ for some }c_1,c_2,C>0\nonumber\\
			&\geq C\,  \exp\left(\pi a k^4 -c\,|k|\right),
		\end{align*}
		which is \eqref{m_1bound3}.
	\end{proof}
	\begin{proposition}[Construction of $M_1$]
		There exists an entire functions $M_1$ of exponential type at most $a\pi$ which satisfies the following
		\begin{align*}
			|M_1(x)|&\leq C_4 |x| e^{-\sqrt{2}\pi \sqrt{\abs{x}}}, \\
		 |M_1(-i\mu_k^+)|&\geq C_5 \, e^{a\pi k^2 -(5+2\sqrt{2})\pi|k|},k \in \z\setminus\{0\},\\
		|M_1(-i\mu_k^-)| &\geq C_6 e^{a\pi k^4-(2\sqrt{2}+1)\pi k^2-8|k|},k \in \z\setminus\{0\},
		\end{align*}
		for some $C_4,C_5,C_6>0$.
	\end{proposition}
	\begin{proof}
		Define 
		\begin{align*}
		s_1(t)=at-b_1\sqrt{t}, \hspace{5 mm} t>0.
		\end{align*}
		Note that $s_1(B_1)=0$ for $B_1=\left(\frac{b_1}{a}\right)^2$ and $s_1(t)$ is increasing for $t >\left(\frac{b}{2a}\right)^2=\frac{B_1}{4}$ . Now define
		\begin{align*}
		\nu_1(t) = \begin{cases} 0, & t\leq B_1 ,\\
		s_1(t), & t\geq B_1 .\end{cases}
		\end{align*}
		Observe that $\nu_1$ is non-negative and non-decreasing function. Moreover, $0\leq \nu_1(t) \leq at$ for $t>0$. 
		Let us define 
		\begin{align*}
		g_1(z)&=\int_{0}^{\infty} \log\left(1-\frac{z^2}{t^2}\right)\, d\nu_1(t) = \int_{B_1}^{\infty} \log\left(1-\frac{z^2}{t^2}\right)\, ds_1(t), \hspace{2mm}  t\in \cplx \setminus \rea, \\
		U_1(z)&= \int_{0}^{\infty} \log\left|1-\frac{z^2}{t^2}\right|\, d\nu_1(t) = \int_{B_1}^{\infty} \log\left|1-\frac{z^2}{t^2}\right|\, ds_1(t), \hspace{2mm}  t\in \cplx .
		\end{align*}
		Similar to the previous proof, we atomize the measure and define
		\begin{align*}
		\tilde{g}_1(z)&=\int_{0}^{\infty} \log\left(1-\frac{z^2}{t^2}\right)\, d[\nu_1(t)] = \int_{B_1}^{\infty} \log\left(1-\frac{z^2}{t^2}\right)\, d[s_1(t)], \hspace{2mm}  t\in \cplx \setminus \rea .\\
		\tilde{U}_1(z)&= \int_{0}^{\infty} \log\left|1-\frac{z^2}{t^2}\right|\, d[\nu_1(t)] = \int_{B_1}^{\infty} \log\left|1-\frac{z^2}{t^2}\right|\, d[s_1(t)], \hspace{2mm}  t\in \cplx .
		\end{align*}
		and then define multiplier as $M_1(z)=\exp(\tilde{g}_1(z-i))$. So, we have $|M_1(z)|=\exp(\tilde{U}_1(z-i))$.
		
		Now, from (\cite{LR14},Proposition 3.9), we get the following estimate of $\tilde{U}_1$ on $\cplx\setminus\rea$:
		\begin{align*}
		-C-b_1\pi\left(1+\frac{1}{\sqrt{2}}\right)\sqrt{|y|} -\log^{+}\frac{|x|}{|y|}& - \log^{+}\left(\frac{x^2+y^2}{B_1^2}\right)-\log2\\ \nonumber & \leq \tilde{U}_1(x+iy)+b_1\pi\sqrt{|x|}-a\pi|y|\leq C + \log^{+}\frac{|x|}{|y|}.
		\end{align*}
		where $C > 0$ is 
		a positive constant independent of $z$ and $k$.
		Using this estimate we get:
		\begin{align*}
		|\tilde{U}_1(x-i)|\leq C + \log^{+}|x| -b_1\pi\sqrt{|x|}+a\pi,
		\end{align*}
		which gives
		\begin{align*}
		|M_1(x)|\leq C |x|e^{-\sqrt{2}\pi\sqrt{|x|}}.
		\end{align*}
		Next we estimate $M_1$ at $\{-i\mu_k^\pm\}_{k\in\z\setminus\{0\}}$,
		\begin{align*}
		|M_1(-i\mu_k^{+})|&=exp\left(\tilde{U}_1\left((-k-O(|k|^{-1}))+i(k^2-1-O(|k|^{-2}))\right)\right)\\
		&\geq C \exp\left( -\sqrt{2}\pi(\sqrt{|k^2-1-O(k^{-2})|}+\sqrt{|k+O(k^{-1})|})-\pi\sqrt{|k^2-1-O(k^{-2})|}\right)\\&\hspace{50 mm} \exp\left(-\log(c\,k^4)+a\pi|k^2-1-O(k^{-2})|\right),\,c,C>0\\
		& \geq C \, \exp(-2\sqrt{2}\pi |k|-\pi|k|-4 \log|c\,k|+a\pi(k^2-1-O(k^{-2})))\nonumber\\
		& \geq C \, \exp(a\pi k^2 -(5+2\sqrt{2})\pi|k|), \text{ and }
		\end{align*}
		\begin{align*}
		|M_1(-i\mu_k^-)|&=exp\left(\tilde{U}_1(-i\mu_k^--i)\right)\\&=\tilde{U}_1((k^3-O(k^{-1}))+i(k^4-k^2-1+O(k^{-2}))) \\
		&\geq C \exp\left( -\sqrt{2}\pi\left(\sqrt{|k^4-2k^2-1+O(k^{-2})|}+\sqrt{|k^3-O(k^{-1})|}\right)+a\pi|k^4-k^2-1-O(|k|^{-2})|\right)\\ &\hspace{50 mm}\exp(-\pi\sqrt{|k^4-2k^2-1+O(k^{-2})|}-\log(c\, k^8)),\,c,C>0\\
		&\geq C\,\exp(a\pi k^4-8\,\log|c\,k|-2\sqrt{2}\pi k^2-\pi k^2)\nonumber\\
		&\geq C\, \exp(a\pi k^4-(2\sqrt{2}+1)\pi k^2-8|k|). 
		\end{align*}
	\end{proof}

	\noindent
	Combining the last two propositions, we get the following result concerning the multiplier $M$ of $P:$
	\begin{proposition}[Multiplier, $M$]
		Let $M(z)=M_1(z)M_2(z)$. Then $M$ is an entire function of exponential type at most $2a\pi=\frac{T}{2}$ and satisfies the following bound:
		\begin{align*}
		|M(x)|&\leq C\,|x|^2e^{-\sqrt{2}\pi \sqrt{\abs{x}}-4\sqrt{2}\pi\sqrt[4]{|x|}},\\
		|M(-i\mu_k^+)|& \geq \widetilde{C} \, e^{2a\pi k^2 -3(3+2\sqrt{2})\pi|k|},\,k\in\z\setminus\{0\},\\
		|M(-i\mu_k^-)|& \geq \widehat{C}\, e^{2a\pi k^4-(2\sqrt{2}+1)\pi k^2-c|k|},\,k\in\z\setminus\{0\},
		\end{align*}
		for some $c,C,\widetilde{C},\widehat{C}>0$.
	\end{proposition} 
	\subsection{Proof of \Cref{prop1}}
	For $k \in \zahl\setminus\{0\}$, define
	\begin{align*}
	\Psi_k^{\pm}(z)=\frac{P(z)}{P'(-i\mu_k^{\pm})(z+i\mu_k^{\pm})} \frac{M(z)}{M(-i\mu_k^{\pm})},
	\end{align*} and
	\begin{align*}
	\Psi_0(z)=\frac{P(z)}{P'(0)z}\frac{M(z)}{M(0)}.
	\end{align*}
	Clearly, each element of $\{\Psi_k^\pm,\Psi_0\}$ is an entire function of exponential type at most $\frac{T}{2}$ and satisfies
	\begin{align*}
	\Psi_k^{\pm}(-i\mu_l^{\pm}&)=\delta_{kl}\,\delta_{\pm},\hspace{2 mm} \forall\, l \in \zahl \text{ and } k \in \zahl\setminus\{0\}, \\
	\Psi_0(-i\mu_l^{\pm})&=\delta_{0l},\hspace{2 mm} \forall\, l \in \zahl.
	\end{align*}
For $k \in \zahl\setminus\{0\}$, we have
\begin{align*}
	|\Psi_k^+(x)| &\leq C \frac{|k|^3 |x| e^{-2\sqrt{2}\pi|x|^{\frac{1}{4}}}}{|x-k|+k^2}e^{-\frac{T}{2} k^2 -2\pi |k|^{\frac{1}{2}}+3(3+2\sqrt{2})\pi|k|},
	\nonumber \\
	&\leq C \frac{|k|^3 }{|x-k|+k^2}e^{-\frac{T}{2} k^2 -2\pi |k|^{\frac{1}{2}}+3(3+2\sqrt{2})\pi|k|}.\\
	|\Psi_k^-(x)| &\leq \widetilde{C} \frac{|k|^7|x|e^{-2\sqrt{2}\pi|x|^{\frac{1}{4}}}}{|x+k^3|+|k^4-2k^2|}e^{-\frac{T}{2} k^4 +(2\sqrt{2}+1)\pi |k|^2+(c-3)\pi|k|},\text{ for some }c>0\nonumber \\
	&\leq \widetilde{C} \frac{|k|^7}{|x+k^3|+|k^4-2k^2|}e^{-\frac{T}{2} k^4 +(2\sqrt{2}+1)\pi |k|^2+(c-3)\pi|k|}.
\end{align*}
These estimates show that $\Psi_k^+,\Psi_k^-\in L^2(\rea)$ with
\begin{align*}
	||\Psi_k^+||_{L^2(\rea)}&\leq C\,|k|^2e^{-\frac{T}{2} k^2 -2\pi |k|^{\frac{1}{2}}+3(3+2\sqrt{2})\pi|k|},\\
	||\Psi_k^-||_{L^2(\rea)}&\leq C |k|^5 e^{-\frac{T}{2} k^4 +(2\sqrt{2}+1)\pi |k|^2+(c-3)\pi|k|}.
\end{align*} 
Similarly, \begin{align*}
	|\Psi_0(x)| \leq Ce^{-2\sqrt{2}\pi|x|^{\frac{1}{4}}},
\end{align*} and so $\Psi_0 \in L^2(\rea)$ as well.

	\noindent
	As discussed in the beginning of this section, we define the biorthogonal family element, $\Theta_k^{\pm}$ and $\Theta_0$ as the inverse Fourier transform of $\Psi_k^{\pm}$ and $\Psi_0$ respectively, where $k \in \zahl\setminus \{0\}$. Then,  $\Theta_k^{\pm},\Theta_0 \in L^2(\rea)$ with
	\begin{align*}
	||\Theta_k^+||_{L^2(\rea)}&\leq C\,|k|^2e^{-\frac{T}{2} k^2 -2\pi |k|^{\frac{1}{2}}+3(3+2\sqrt{2})\pi|k|},\\
	||\Theta_k^-||_{L^2(\rea)}&\leq C |k|^5 e^{-\frac{T}{2} k^4 +(2\sqrt{2}+1)\pi |k|^2+(c-3)\pi|k|},\\
	||\Theta_0||_{L^2(\rea)} &\leq C,
	\end{align*}
	where $c,C>0$.

\section{Appendix A. Proof of \Cref{lemma glass}}\label{app}
\begin{proof}
	First we write the expression of $U_2$ \eqref{u2} restricted on the real line :
	\begin{align*}
	U_2(x)= \int_{0}^{\infty} \log\left|1-\frac{x^2}{t^2}\right|\, d s_2(t)-\int_{0}^{B} \log\left|1-\frac{x^2}{t^2}\right|\, d s_2(t), \hspace{2mm}  x\in \rea .
	\end{align*}
	Thus using \eqref{eq:log}, we obtain
	\begin{align*}
	U_2(x)+b_2\,\pi\, \cot\left(\frac{\pi}{8}\right)\sqrt[4]{|x|}= -\int_{0}^{B} \log\left|1-\frac{x^2}{t^2}\right|\, d s_2(t), \hspace{2mm}  x\in \rea .
	\end{align*}
	Applying a change of variable $t\to \frac{t}{B}$, we have 
	\begin{align*}
	\int_{0}^{B} \log\left|1-\frac{x^2}{t^2}\right|\, d s_2(t)=aB_2 \int_{0}^{1} \log\left|1-\frac{x^2}{B_2^2 t^2}\right|\, d(t-t^{1/4}).
	\end{align*}
	Let us define the following function in $(0,\infty)$ 
	\begin{align*}
	\theta(x)=\int_{0}^{1} \log\left|1-\frac{x^2}{t^2}\right|\, d(t-t^{1/4}).
	\end{align*}
	Note that $\theta\in L^{\infty}(0, \infty)$, see \cite{AM21} (Lemma 16) for details. A direct computation conclude the proof of \Cref{lemma glass}. 
\end{proof}
\section{Appendix B. Details of the inequality
	\eqref{ineq}}\label{secinq}
Let us first denote 
\begin{align*}
	\mathcal{J}=\int_{-\infty}^{\infty}\frac{|s|^{\frac{1}{4}}}{1+|x-s|^2}\,ds.
\end{align*}
Applying the change of variable $x-s =\tan(w)$, we have
\begin{align}\label{id}
	\mathcal{J} =\int_{-\frac{\pi}{2}}^{\frac{\pi}{2}} |x-\tan(w)|^{1/4}dw.
\end{align}
Next, we put $x=\tan\alpha$ in \eqref{id} for some $\alpha\in (-\frac{\pi}{2}, \frac{\pi}{2})$. Then the above identity \eqref{id} becomes 
\begin{align}\label{I}
	\mathcal{J}=|\sec^{1/4}\alpha|\int_{-\frac{\pi}{2}}^{\frac{\pi}{2}} \frac{|\sin^{1/4}(w-\alpha)|}{|\cos^{1/4}(w)|}dw.
\end{align}
Next, we will use the following inequalities:
\begin{align*}
	0\leq |\sin^{1/4}(w-\alpha)|\leq 1,\quad  1\leq \frac{1}{|\cos^{1/4}(w)|}<\infty,\,\, & w\in \left(-\frac{\pi}{2}, \frac{\pi}{2}\right).
\end{align*}
Thus, from \eqref{I} we have, 
\begin{align}\label{ineqI}
\left(1+x^2\right)^{\frac{1}{8}}	\int_{-\frac{\pi}{2}}^{\frac{\pi}{2}}	|\sin^{1/4}(w-\alpha)|\, dw \leq\mathcal{J}\leq \left(1+x^2\right)^{\frac{1}{8}} \int_{-\frac{\pi}{2}}^{\frac{\pi}{2}}\frac{1}{|\cos^{1/4}(w)|} \,dw.
\end{align}
Now, \begin{align}\label{eqI}
	\int_{-\frac{\pi}{2}}^{\frac{\pi}{2}}	|\sin^{1/4}(w-\alpha)|\, dw=\int_{-\frac{\pi}{2}+\alpha}^{\frac{\pi}{2}+\alpha}|\sin^{1/4}(w)|\, dw=2\int_{0}^{\frac{\pi}{2}}|\sin^{1/4}(w)|\, dw.
\end{align} In the last equality we have used periodicity of sine function. Next, as cosine function is even, and positive in $\left(-\frac{\pi}{2}, \frac{\pi}{2}\right)$, we have: \begin{align}\label{eqI1}
	\int_{-\frac{\pi}{2}}^{\frac{\pi}{2}}\frac{1}{|\cos^{1/4}(w)|}dw=2\int_{0}^{\frac{\pi}{2}}{\cos^{-1/4}(w)} \,dw.
\end{align}
Let us recall, $$B(m,n)=2\int_{0}^{\frac{\pi}{2}}\sin^{2m-1}(w) \, \cos^{2n-1}(w) \, \, dw= \frac{\Gamma(m)\Gamma(n)}{\Gamma(m+n)}, \quad m, \,n >0.$$ Using the above formula and \eqref{ineqI}, \eqref{eqI}, \eqref{eqI1}, we have the desired inequality \eqref{ineq}.
\section{Appendix C. Proof of well-posedness results}\label{sect7}
\subsection{Proof of \Cref{adj well-prop}}
\begin{proof}
	Let us first perform a change of variable $t\mapsto T-t$ in \eqref{adjoint} to get the forward adjoint system
	\begin{equation}\label{for}
		\begin{cases}
	\varphi_t+\varphi_{xxxx}-\varphi_{xxx}+\varphi_{xx}+\psi_x=h_1,\quad\quad& (t,x) \in (0,T)\times (0,2\pi),\\
	\psi_t-\psi_{xx}-\psi_x+\varphi_x=h_2,\quad\quad & (t,x) \in (0,T)\times (0,2\pi),\\
	\varphi(0,x)=\varphi_T(x), \psi(0,x)=\psi_T(x),\quad\quad & x\in (0,2\pi),
	\end{cases}
	\end{equation}
	with the same periodic boundary condition as in \eqref{adjoint}.
	
	Let us assume the data $h_1, h_2, \varphi_T\text{ and } \psi_T$ to be regular enough. Then we multiply the first and second equations of the system \eqref{for} by $\overline\varphi$ and $\overline\psi$ respectively, and then take real part of the equation. Performing integration by parts and then adding them, we get
	\begin{align}
	&\nonumber\frac{1}{2}\dfrac{d}{dt}\int_0^{2\pi}(\,|\varphi|^2+|\psi|^2)+\int_0^{2\pi}|\varphi_{xx}|^2+\int_0^{2\pi}|\psi_{x}|^2\\&\hspace{50 mm}\nonumber=\Re\left(\int_0^{2\pi} h_1\overline\varphi\right)+\Re\left(\int_0^{2\pi}h_2\overline\psi\right)-\Re\left(\int_0^{2\pi}\varphi_{xx}\overline\varphi\right)\\
	&\nonumber\hspace{50 mm}\leq \int_0^{2\pi} |h_1\overline\varphi|+\int_0^{2\pi}|h_2\overline\psi|+\int_0^{2\pi}|\varphi_{xx}\overline\varphi|.
	\end{align}
	Using Young's inequality in the last term of R.H.S, we get
	\begin{align}\label{ineq0}
	\dfrac{d}{dt}\int_0^{2\pi}(\,|\varphi|^2+|\psi|^2)+\int_0^{2\pi}|\varphi_{xx}|^2+2\int_0^{2\pi}|\psi_{x}|^2\leq 2\int_0^{2\pi} |h_1\overline\varphi|+2\int_0^{2\pi}|h_2\overline\psi|+\int_0^{2\pi}|\varphi|^2.
	\end{align}

	Next we multiply the first equation of the above forward adjoint system \eqref{for} by $\overline\varphi_{xxxx}$ and then consider the real part of the equation. Performing integration by parts, we obtain
	\begin{align*}
	&\Re\left(\int_0^{2\pi}\varphi_{txx}\overline\varphi_{xx}\right)+\int_0^{2\pi}|\varphi_{xxxx}|^2\\&\hspace{30 mm}=-\Re\left(\int_0^{2\pi}\varphi_{xx}\overline \varphi_{xxxx}\right)-\Re\left(\int_0^{2\pi}\psi_x\overline\varphi_{xxxx}\right)+\Re\left(\int_0^{2\pi}h_1\overline\varphi_{xxxx}\right)\\
	&\hspace{30 mm}\leq\int_0^{2\pi}|\varphi_{xx}\overline\varphi_{xxxx}|+\int_0^{2\pi}|\psi_x\overline\varphi_{xxxx}|+\Re\left(\int_0^{2\pi}h_1\overline\varphi_{xxxx}\right).
	\end{align*}
	Applying Young's inequality for the first and second terms of R.H.S with $\epsilon>0$, we have:
	\begin{align*}&\frac{1}{2}\dfrac{d}{dt}\int_0^{2\pi}|\varphi_{xx}|^2+\int_0^{2\pi}|\varphi_{xxxx}|^2\leq \frac{\epsilon}{2}\int_{0}^{2\pi}|\varphi_{xxxx}|^2+\frac{1}{2\epsilon}\int_{0}^{2\pi}|\varphi_{xx}|^2+\frac{\epsilon}{2}\int_{0}^{T}|\varphi_{xxxx}|^2 \\\nonumber&\hspace{90 mm}+\frac{1}{2\epsilon}\int_{0}^{2\pi}|\psi_{x}|^2+\Re\left(\int_0^{2\pi}h_1\overline\varphi_{xxxx}\right).\end{align*}
	After simplifying, we deduce	\begin{align}\label{eq1}&\dfrac{d}{dt}\int_0^{2\pi}|\varphi_{xx}|^2+2(1-\epsilon)\int_0^{2\pi}|\varphi_{xxxx}|^2\leq C\left(\,\int_{0}^{2\pi}|\varphi_{xx}|^2+\int_{0}^{2\pi}|\psi_{x}|^2\right)+2\,\Re\left(\int_0^{2\pi}h_1\overline\varphi_{xxxx}\right).
	\end{align}
	Next multiply the second equation of the adjoint \eqref{for} by $\overline\psi_{xx}$, so that for any $\epsilon>0$, we have:
	\begin{align}\label{eq2}
	\nonumber \frac{d}{dt}\int_{0}^{2\pi}|\psi_x|^2+2(1-\epsilon)\int_0^{2\pi}|\psi_{xx}|^2&\leq C\left(\,\int_{0}^{2\pi}|\psi_x|^2+\int_0^{2\pi}|\varphi_x|^2\right)+2\,\Re\left(\int_{0}^{2\pi}h_2\overline\psi_{xx}\right)\\
	&\leq C\left(\,\int_{0}^{2\pi}|\psi_x|^2+\int_0^{2\pi}|\varphi_{xx}|^2\right)+2\,\Re\left(\int_{0}^{2\pi}h_2\overline\psi_{xx}\right).
	\end{align}	
	On adding equations \eqref{eq1} and \eqref{eq2}, we get:
	\begin{align}
	\nonumber &\dfrac{d}{dt}\int_0^{2\pi}(\,|\varphi_{xx}|^2+|\psi_x|^2)+2(1-\epsilon)\int_0^{2\pi}(|\varphi_{xxxx}|^2+|\psi_{xx}|^2)\leq C\left(\,\int_{0}^{2\pi}|\varphi_{xx}|^2+\int_{0}^{2\pi}|\psi_{x}|^2\right)\\
	\label{ineq1}&\hspace{90 mm} +2\,\Re\left(\int_0^{2\pi}h_1\overline\varphi_{xxxx}\right)+2\,\Re\left(\int_{0}^{2\pi}h_2\overline\psi_{xx}\right).
	\end{align}
	\item[Case 1: $(h_1,h_2)\in L^2(0,T;L^2(0,2\pi)).$] 
	Using Young's inequality in R.H.S of \eqref{ineq0}, we obtain
	\begin{align*} \dfrac{d}{dt}\int_0^{2\pi}(\,|\varphi|^2+|\psi|^2)+\int_0^{2\pi}|\varphi_{xx}|^2+2\int_0^{2\pi}|\psi_{x}|^2\leq 2\int_0^{2\pi}(|\varphi|^2+|\psi|^2)+\int_0^{2\pi} \left(|h_1|^2+|h_2|^2\right).
	\end{align*}
	Multiplying the inequality by $e^{-3t}$ and then integrating over $[0,s]\subset [0,T]$, we get:
	\begin{align}\label{new}
	\nonumber\int_0^{2\pi}\left(|\varphi(s,\cdot)|^2+|\psi(s,\cdot)|^2\right)+&\int_0^s\int_0^{2\pi}(|\varphi|^2+|\psi|^2)\\&\leq C \left(\int_0^T\int_0^{2\pi} (|h_1|^2+|h_2|^2)+\int_0^{2\pi}|\varphi_T|^2+\int_0^{2\pi}|\psi_T|^2\right).
	\end{align}
Using the inequality $\Re(z)\leq |z|$, for $z\in\cplx$ and Young's inequality in the last two integrals of \eqref{ineq1}, we get
	\begin{align}\label{eq}
	\nonumber&\dfrac{d}{dt}\int_0^{2\pi}(\,|\varphi_{xx}|^2+|\psi_x|^2)+\int_0^{2\pi}(\,|\varphi_{xxxx}|^2+|\psi_{xx}|^2)\\&\hspace{30 mm}\leq C\left(\,\int_{0}^{2\pi}|\varphi_{xx}|^2+\int_{0}^{2\pi}|\psi_{x}|^2\right) +\int_{0}^{2\pi}(\,|h_1|^2+|h_2|^2).
	\end{align}
	Multiplying \eqref{eq} by $e^{-Ct}$ and then integrating w.r.t $t$ over $[0,s]\subset [0,T]$, for all $s\in [0,T]$ we get
	\begin{align}\label{first}
	\int_0^{2\pi}(\,|\varphi_{xx}(s,\cdot)|^2+|\psi_x(s,\cdot)|^2)\leq C\left(\,\int_0^T\int_{0}^{2\pi}(|h_1|^2+|h_2|^2)+\int_0^{2\pi}(\,|(\varphi_T)_{xx}|^2+|(\psi_T)_x|^2)\right).
	\end{align}
	Now integrating \eqref{eq} on $[0,s]\subset [0,T]$, and using \eqref{first}, we obtain
	\begin{align}\label{second}
	\int_0^s\int_{0}^{2\pi}|(\varphi_{xxxx}|^2+|\psi_{xx}|^2)\leq C\left(\int_0^T\int_{0}^{2\pi}(|h_1|^2+|h_2|^2)+\int_0^{2\pi}(\,|(\varphi_T)_{xx}|^2+|(\psi_T)_x|^2)\right).
	\end{align}
	Adding the inequalities \eqref{new}, \eqref{first} and \eqref{second} to get
	\begin{align*}
	&\int_0^{2\pi}(|\varphi(s,\cdot)|^2+|\psi(s,\cdot)|^2)+\int_0^{2\pi}(\,|\varphi_{xx}(s,\cdot)|^2+|\psi_x(s,\cdot)|^2)+\int_0^s\int_0^{2\pi}(|\varphi|^2+|\psi|^2)\\
	&\quad \quad\quad+\int_0^s\int_{0}^{2\pi}|(\varphi_{xxxx}|^2+|\psi_{xx}|^2)\leq C\left( \int_0^T\int_0^{2\pi} (|h_1|^2+|h_2|^2)+||\varphi_T||^2_{H^2_{per}}+||\psi_T||^2_{H^1_{per}}\right).
	\end{align*}
	On taking supremum over $s \in[0,T]$ and using the equivalence of sobolev norms, we get 
	\begin{align*}
	||(\varphi,\psi)||_{L^2(0,T;H^4\times H^2)\cap L^{\infty}(0,T;H_{per}^2\times H^1_{per})}\leq C\left( ||(h_1,h_2)||_{(L^2(0,T;L^2(0,2\pi)))^2}+||(\varphi_T,\psi_T)||_{H_{per}^2\times H_{per}^1}\right).
	\end{align*}
	\item[Case 2: $(h_1,h_2)\in L^1(0,T;H_{per}^2(0,2\pi)\times H^1_{per}(0,2\pi))$.] 
	From \eqref{ineq0}, we have:
	\begin{align*}
		\nonumber\dfrac{d}{dt}\int_0^{2\pi}(\,|\varphi|^2+|\psi|^2)&+\int_0^{2\pi}|\varphi|^2+\int_0^{2\pi}|\psi|^2+\int_0^{2\pi}|\varphi_{xx}|^2+2\int_0^{2\pi}|\psi_{x}|^2\\&\quad\quad\leq 2\left( \int_0^{2\pi}|\varphi|^2+\int_0^{2\pi}|\psi|^2\right)+2\int_0^{2\pi} |h_1\overline\varphi|+2\int_0^{2\pi}|h_2\overline\psi|	.
	\end{align*}
	Multiplying the inequality by $e^{-2t}$ and then integrating w.r.t over $[0,s]\subset[0,T]$, we write
	\begin{align}
	\nonumber&\int_0^{2\pi}(\,|\varphi(s,\cdot)|^2+|\psi(s,\cdot)|^2)+\int_0^s\intg |\varphi|^2+\int_0^s\intg |\psi|^2\\\nonumber&\hspace{33mm}\leq C\left(\int_0^T\int_0^{2\pi} |h_1\overline\varphi|+\int_0^T\int_0^{2\pi}|h_2\overline\psi|+\intg(|\varphi_T|^2+|\psi_T|^2)\right)\\&\hspace{33 mm}\leq C\left(||h_1||_{L^1(L^2)}||\varphi||_{L^{\infty}(L^2)}+||h_2||_{L^1(L^2)}||\psi||_{L^{\infty}(L^2)}+\intg(|\varphi_T|^2+|\psi_T|^2)\right).
	\end{align}
	Taking supremum in $s$ over $[0,T]$, we obtain
	\begin{align}\label{main}
		\nonumber||\varphi||^2_{L^{\infty}(L^2)}+||\psi||^2_{L^{\infty}(L^2)}&+||\varphi||^2_{L^2(L^2)}+||\psi||^2_{L^2(L^2)}\\&\leq C\left(||h_1||_{L^1(L^2)}||\varphi||_{L^{\infty}(L^2)}+||h_2||_{L^1(L^2)}||\psi||_{L^{\infty}(L^2)}+\intg(|\varphi_T|^2+|\psi_T|^2)\right).
	\end{align}
	Note that 
	\begin{align*}
		&||h_1||_{L^1(L^2)}||\varphi||_{L^{\infty}(L^2)}+||h_2||_{L^1(L^2)}||\psi||_{L^{\infty}(L^2)}\\&\hspace{40mm}\leq \left(||h_1||_{L^1(L^2)}+||h_2||_{L^1(L^2)}\right)\left(||\varphi||_{L^{\infty}(L^2)}+||\psi||_{L^{\infty}(L^2)}\right),\\
		\text{and}& \left(\intg|\varphi_T|^2+|\psi_T|^2\right)^{\frac{1}{2}}\leq ||\varphi||_{L^{\infty}(L^2)}+||\psi||_{L^{\infty}(L^2)}.
	\end{align*}
	So, using these in the last inequality \eqref{main}, we get
	\begin{align}\label{first0}
		&\nonumber\left(||\varphi||_{L^{\infty}(L^2)}+||\psi||_{L^{\infty}(L^2)}\right)^2
		\\\nonumber&\hspace{17 mm}\leq C \left(||\varphi||_{L^{\infty}(L^2)}+||\psi||_{L^{\infty}(L^2)}\right)\left(||h_1||_{L^1(L^2)}+||h_2||_{L^1(L^2)}+||\varphi_T||_{L^2}+||\psi_T||_{L^2}\right), \\
		\text{thus, }&\,\,||\varphi||_{L^{\infty}(L^2)}+||\psi||_{L^{\infty}(L^2)}\leq C\left(\, ||h_1||_{L^1(L^2)}+||h_2||_{L^1(L^2)}+||\varphi_T||_{L^2}+||\psi_T||_{L^2}\right).
	\end{align}
	From \eqref{main} and using the relation $a^2+b^2\leq (a+b)^2\leq 2(a^2+b^2)$ for $a,b\geq0$, we also have
	\begin{align}\label{first1}
		\nonumber&\left(||\varphi||_{L^2(L^2)}+||\psi||_{L^2(L^2)}\right)^2\\&\quad\quad\quad\quad\leq\nonumber C\left(||h_1||_{L^1(L^2)}||\varphi||_{L^{\infty}(L^2)}+||h_2||_{L^1(L^2)}||\psi||_{L^{\infty}(L^2)}+\intg(|\varphi_T|^2+|\psi_T|^2)\right)\\\nonumber
		&\quad\quad\quad\quad\leq C\left(||h_1||_{L^1(L^2)}+||h_2||_{L^1(L^2)}||\right)\left(||\varphi||_{L^{\infty}(L^2)}+||\psi||_{L^{\infty}(L^2)}\right)+\left(||\varphi||^2_{L^{\infty}(L^2)}+||\psi||^2_{L^{\infty}(L^2)}\right)\\&\nonumber
		\quad\quad\quad\quad\leq C\left(||\varphi||_{L^{\infty}(L^2)}+||\psi||_{L^{\infty}(L^2)}\right) \left(||\varphi||_{L^{\infty}(L^2)}+||\psi||_{L^{\infty}(L^2)}+||h_1||_{L^1(L^2)}+||h_2||_{L^1(L^2)}||\right)\\&\quad\quad\quad\quad\nonumber
		\leq C\left(\, ||h_1||_{L^1(L^2)}+||h_2||_{L^1(L^2)}+||\varphi_T||_{L^2}+||\psi_T||_{L^2}\right)^2\\&
		\text{and so, } ||\varphi||_{L^2(L^2)}+||\psi||_{L^2(L^2)}\leq C\left(\, ||h_1||_{L^1(L^2)}+||h_2||_{L^1(L^2)}+||\varphi_T||_{L^2}+||\psi_T||_{L^2}\right)
	\end{align}
	Thus combining \eqref{first0} and \eqref{first1}, we have
	\begin{align}\label{first2}
		\nonumber ||\varphi||_{L^{\infty}(L^2)}+||\psi||_{L^{\infty}(L^2)}+||\varphi||_{L^2(L^2)}&+||\psi||_{L^2(L^2)}\\&\leq C\left(\, ||h_1||_{L^1(L^2)}+||h_2||_{L^1(L^2)}+||\varphi_T||_{L^2}+||\psi_T||_{L^2}\right)
	\end{align}

	Performing integration by parts in \eqref{ineq1}, we get:
	\begin{align}\label{eq0}
	\nonumber \dfrac{d}{dt}\int_0^{2\pi}(\,|\varphi_{xx}|^2+&|\psi_x|^2)+\int_0^{2\pi}|\varphi_{xxxx}|^2+\int_0^{2\pi}|\psi_{xx}|^2\\& \nonumber \leq C\left(\,\int_{0}^{2\pi}|\varphi_{xx}|^2+\int_{0}^{2\pi}|\psi_{x}|^2\right)
	+2\,\Re\left(\int_{0}^{2\pi}(h_1)_{xx}\varphi_{xx}\right)+2\,\Re\left(\int_{0}^{2\pi}(h_2)_x\psi_{x}\right)\\&\leq C\left(\,\int_{0}^{2\pi}|\varphi_{xx}|^2+\int_{0}^{2\pi}|\psi_{x}|^2\right)
	+2\int_{0}^{2\pi}|(h_1)_{xx}\varphi_{xx}|+2\int_{0}^{2\pi}|(h_2)_x\psi_{x}|.
	\end{align}
	Again multiplying the inequality by $e^{-Ct}$,  integrating w.r.t $t$ over $[0,s]\subset [0,T]$ and then taking supremum over $s\in [0,T]$, we get
	\begin{align}\label{second0}
	\nonumber &\sup_{s\in[0,T]}\int_0^{2\pi}(\,|\varphi_{xx}(s,\cdot)|^2+|\psi_x(s,\cdot)|^2)\\ &\nonumber\hspace{15mm}\leq C\left(\int_0^T\int_{0}^{2\pi}|(h_1)_{xx}\varphi_{xx}|+\int_0^T\int_{0}^{2\pi}|(h_2)_x\psi_{x}|\right)
	+\int_0^{2\pi}(\,|(\varphi_T)_{xx}|^2+|(\psi_T)_x|^2)\\\nonumber
	&\hspace{15mm}\leq C\left(||(h_1)_{xx}||_{L^1(L^2)}||(\varphi)_{xx}||_{L^{\infty}(L^2)}+||(h_2)_{x}||_{L^1(L^2)}||(\psi)_{x}||_{L^{\infty}(L^2)}\right)\\&\hspace{72 mm}+\int_0^{2\pi}(\,|(\varphi_T)_{xx}|^2+|(\psi_T)_x|^2).		\end{align}	
	Using similar analysis as done above, we get:
	\begin{align}\label{second1}
	||\varphi_{xx}||_{L^{\infty}(L^2)}+||\psi_{x}||_{L^{\infty}(L^2)}\leq C(\,||(h_1)_{xx}||_{L^1(L^2)}+||(h_2)_{x}||_{L^1(L^2)}+||(\varphi_T)_{xx}||_{L^2}+||(\psi_T)_x||_{L^2}).
	\end{align}			
	Integrating \eqref{eq0} w.r.t $t$ over $[0,T]$ and using inequality \eqref{second1}, we get
	\begin{align}\label{second2}
	||\varphi_{xxxx}||_{L^2(L^2)}+||\psi_{xx}||_{L^2(L^2)}\leq C(\,||(h_1)_{xx}||_{L^1(L^2)}+||(h_2)_{x}||_{L^1(L^2)}+||(\varphi_T)_{xx}||_{L^2}+||(\psi_T)_x||_{L^2}).
	\end{align}	
	On adding \eqref{first1}, \eqref{second1} and \eqref{second2}, we get:
	\begin{align*}
	||(\varphi,\psi)||_{L^2(0,T;H^4\times H^2)\cap L^{\infty}(0,T;H_{per}^2\times H^1_{per})}\leq C\left( ||(h_1,h_2)||_{L^1(0,T;H_{per}^2\times H^1_{per})}+||(\varphi_T,\psi_T)||_{H_{per}^2\times H_{per}^1}\right).
	\end{align*}
	
	Thus, combining both cases together we have
	\begin{align*}
	\left|\left|(\varphi,\psi)\right|\right|_{L^{\infty}([0,T],H_{per}^2\times H_{per}^1)\cap L^2(0,T;H^4\times H^2)} \leq C\left(||(h_1,h_2)||_{\mathbf{X}}+\left|\left|(\varphi_T,\psi_T)\right|\right|_{H_{per}^2\times H_{per}^1}\right), 
	\end{align*}
	for some $C>0$.
	
	The solution $(\varphi,\psi)\in L^2(0,T;H^4\times H^2)$, so from the equation we get $(\varphi_t,\psi_t)\in L^2(0,T;L^2\times L^2)$ and hence by the classical properties of these spaces we get $(\varphi,\psi)\in C([0,T];H^2_{per}\times H_{per}^1)$ and hence the proof is complete.
\end{proof}
	
\bibliographystyle{plain}
\bibliography{Stabilized_KS_moment_.bib}
\end{document}